\documentclass{amsart}
\usepackage{amssymb}
\usepackage{color}

\newtheorem{theorem}{Theorem}[section]
\newtheorem{example}[theorem]{Example}
\newtheorem{lemma}[theorem]{Lemma}
\newtheorem{proposition}[theorem]{Proposition}
\newtheorem{problem}[theorem]{Problem}
\newtheorem{corollary}[theorem]{Corollary}

\theoremstyle{definition}
\newtheorem{definition}[theorem]{Definition}
\newtheorem{remark}[theorem]{Remark}

\newcommand{\A}{\mathcal A}
\newcommand{\IR}{\mathbb R}
\newcommand{\IZ}{\mathbb Z}

\newcommand{\Z}{\mathbb Z}
\newcommand{\V}{\mathcal{V}}
\newcommand{\F}{\mathcal{F}}
\newcommand{\diam}{\operatorname{diam}}
\newcommand{\Ra}{\Rightarrow}
\newcommand{\id}{\mathrm{id}}
\newcommand{\w}{\omega}
\newcommand{\IN}{\mathbb N}
\newcommand{\e}{\varepsilon}

\newcommand{\K}{\mathcal K}

\newcommand{\U}{\mathcal U}
\newcommand{\cl}{\mathrm{cl}}
\newcommand{\mn}{\mathfrak{mn}}
\newcommand{\tmn}{\mathfrak{tmn}}

\newcommand{\mV}{\mathcal V}

\newcommand{\N}{\mathbb{N}}
\newcommand{\E}{\mathcal E}
\newcommand{\W}{\mathcal W}
\newcommand{\D}{\mathcal D}
\newcommand{\G}{\mathcal G}
\newcommand{\pr}{\mathrm{pr}}
\newcommand{\St}{\mathcal{S}t}
\newcommand{\dom}{\mathrm{dom}}

\author{Taras Banakh}

\author{Wies\l aw Kubi\'s}

\author{Natalia Novosad}

\author{Magdalena Nowak}

\author{Filip Strobin}

\address{T.Banakh, W.Kubi\'s, M.Nowak, F.Strobin: Institute of Mathematics, Jan Kochanowski University in Kielce, ul. \'Swietokrzyska 15, 25-406 Kielce, Poland}
\email{t.o.banakh@gmail.com, kubis@math.cas.cz, magdalena.nowak805@gmail.com}
\address{N.Novosad: Istitute for Applied Problems of Mechanics and Mathematics of National Academy of Sciences of Ukraine, Naukova 3b, Lviv, Ukraine}
\email{natalia.kasper@gmail.com}
\address{F.Strobin: Institute of Mathematics, \L\'od\'z University of Technology, W\'olcza\'nska 215, 93-005 \L\'od\'z, Poland}
\email{filip.strobin@p.lodz.pl}
\thanks{The research of Magdalena Nowak has been partially supported by NCN grant DEC-2012/07/N/ST1/03551.}

\title{Contractive function systems, their attractors and metrization}
\date{}

\begin{document}

\begin{abstract} { In this paper we study the Hutchinson-Barnsley theory of fractals in the setting of multimetric spaces (which are sets endowed with point separating families of pseudometrics) and in the setting of topological spaces. We find natural connections between these two approaches.}
\end{abstract}
\maketitle

\section{Introduction}

In this paper we consider topological and contracting properties of iterated function systems, well-known in the Theory of Fractals. By an {\em Iterated Function System} on a topological space $X$ we understand a dynamical system on the hyperspace $\K(X)$ of $X$ generated by a finite family $\F$ of continuous self-maps of $X$. In this case the finite family $\F$ will be called a {\em function system} on $X$.

By $\K(X)$ we denote the space of non-empty compact subsets of $X$ endowed with the Vietoris topology.
This topology is generated by the subbase consisting of the sets
$$\{K\in \K(X):K\cap U\ne\emptyset\}\mbox{ \ and \ }\{K\in\K(X):K\subset U\},$$
where $U$ runs over open sets in $X$. If the topology of the space $X$ is generated by a metric $d$, then the Vietoris topology on $\K(X)$ is generated by the {\em Hausdorff metric}
 $$d_H(A,B)=\max\big\{\max_{a\in A}\min_{b\in B}d(a,b),\,\max_{b\in B}\min_{a\in A}d(a,b)\big\}.$$

By $X^X$ we denote the semigroup of all self-maps of $X$, endowed with the operation of composition.
The identity transformation $\id_X$ of $X$ is the (two-sided) unit in the semigroup $X^X$, so $X^X$ is a monoid. For two subsets $\A,\mathcal B\subset X^X$ let $\A\circ\mathcal B=\{\alpha\circ\beta:\alpha\in\mathcal A,\;\beta\in\mathcal B\}$. Each subset $\F\subset X^X$ generates the submonoid $\F^{<\w}=\bigcup_{n\in\w}\F^n$ of $X^X$ where $\F^0=\{\id_X\}$ and $\F^{n+1}=\F\circ\F^n$ for $n\in\w$. For every $k\in\w$ the subset $\F^{\ge k}=\bigcup_{n\ge k}\F^n$ is a subsemigroup of $X^X$.
For each subsets $\F\subset X^X$ and $A\subset X$ let $\F(A)=\bigcup_{f\in \F}f(A)\subset X$.

Any function system $\F\subset X^X$ (i.e., a finite family of continuous self-maps) on a topological space $X$ induces a continuous map $\F:\K(X)\to\K(X)$ assigning to each compact set $K\in\K(X)$ the compact set $\F(K)=\bigcup_{f\in \F}f(K)$ of $X$.

A non-empty compact set $A_\F\in\K(X)$ is called an {\em attractor} of a function system $\F\subset X^X$ if $\F(A_\F)=A_\F$ and for every compact set $K\in\K(X)$ the sequence $\big(\F^n(K)\big)_{n=1}^\infty$ converges to $A_\F$ in the Vietoris topology on $\K(X)$.
In a Hausdorff topological space $X$ a finite system $\F\subset X^X$ of continuous self-maps has at most one attractor.

Observe that a  function system $\F$ on a topological space $X$ has an attractor if and only if the map $\F:\K(X)\to\K(X)$ has an attracting fixed point. We shall say that a self-map $f:X\to X$ of a topological space $X$ has an {\em attracting fixed point} if there is a point $x_\infty\in X$ such that $f(x_\infty)=x_\infty$ and for every point $x\in X$ the sequence $\big(f^n(x)\big)_{n=1}^\infty$  converges to $x_\infty$. This observation makes possible to apply results of Fixed Point Theory to finding attractors of function systems. In fact, the theory of functions systems with attractors has been deeply studied for the last thirty years -- the great impact on this field had the papers \cite{B} of Barnsley and \cite{Hut} of Hutchinson (who exploited the Banach Fixed Point Theorem), and \cite{Ha} of Hata (who applied some other fixed point theorems).

In this paper we shall be interested in detecting function systems possessing an attractor.
A necessary condition for a function system $\F$ to have an attractor is the compact dominacy. We shall say that a function system $\F$ on a topological space $X$ is {\em compact-dominating} if each compact set $C\subset X$ is contained in a compact set $K\subset X$ such that $\F(K)\subset K$. If a function system $\F$ on $X$ has an attractor $A_\F$, then for every compact set $C\subset X$ the set $K=A_\F\cup\bigcup_{n\in\w}\F^n(C)$ is compact, contains $C$ and has the property $\F(K)\subset K$, which means that $\F$ is compact-dominating.

A quite general topological condition, sufficient for the existence of an attractor, was introduced by Mihail \cite{Mih} who defined a function system $\F\subset X^X$ on a topological space $X$ to be {\em topologically contracting} if $\F$ is compact-dominating and for every non-empty compact set $K\subset X$ with $\F(K)\subset K$ and every sequence $(f_n)_{n\in\w}\in\F^\w$ the intersection $\bigcap_{n\in\w}f_0\circ\dots\circ f_n(K)$ is a singleton.

The last condition allows us to define the map $\pi_K:\F^\w\to X$ assigning to each sequence $\vec f=(f_n)_{n\in\w}\in\F^\w$ the unique point of the singleton $\bigcap_{n\in\w}f_0\circ\dots\circ f_n(K)$.
Clearly, the value $\pi_K(\vec f)$ does not depend
on the choice of $K$ {so, in fact, we can define the mapping $\pi=\pi_K:\F^\w\to X$. Also by the compact dominacy, we can easily see that for any compact set $K$, the sequence $(f_0\circ\dots\circ f_n(K))$ converges to $\{\pi(\vec f)\}$ (with respect to the Vietoris topology). It can be shown that the map $\pi:\F^\w\to X$} is continuous with respect to the Tychonoff product topology on the countable power $\F^\w$ of the finite space $\F$ endowed with the discrete topology and the image $A_\F=\pi(\F^\w)$ is an attractor of $\F$. The attractor $\A_\F$ is compact and metrizable (being the continuous image of the compact metrizable space $\F^\w$). These facts were proved by Mihail in \cite{Mih} (cf. also \cite{D} and \cite{K}).

\begin{theorem}[Mihail]\label{Mih} Each topologically contracting function system $\F$ on a Hausdorff topological space $X$ has an attractor $A_\F$, that can be found as  the image of $\F^\w$ under the map {$\pi:\F^\w\to X$}.
\end{theorem}

A vast source of topologically contracting function systems is given by function systems consisting of contracting maps on metric or multimetric spaces. We shall be interested in six types of contracting maps on metric spaces. For a map $f:X\to Y$ between two metric spaces $(X,d_X)$ and $(Y,d_Y)$
its {\em oscillation} is the map $\w_f:[0,\infty]\to[0,\infty]$ assigning to each number $\delta\in[0,\infty]$ the number
$$\w_f(\delta)=\sup\big\{d_Y(f(x),f(x')):x,x'\in X,\;d_X(x,x')\le \delta\big\}\in[0,\infty].$$
For a natural number $n\in\IN$ by $\w^n_f$ we denote the $n$th iteration of $\w_f$ and by $\w_{f^n}$ the oscillation of the $n$th iteration of $f$.

A self-map $f:X\to X$ defined on a metric space $(X,d)$ is called
\begin{itemize}
\item {\em Banach contracting} if $\sup_{0<t<\infty}\w_f(t)/t<1$;
\item {\em Rakotch contracting} if $\sup_{a\le t<\infty}\w_f(t)/t<1$ for every $a>0$;
\item {\em Krasnoselski\u\i{} contracting} if $\sup_{a\le t\le b}\w_f(t)/t<1$ for every $0<a<b<\infty$;
\item {\em Matkowski contracting} if $\lim_{n\to\infty}\w^n_f(t)=0$ for all $t>0$;
\item {\em eventually contracting} if $\lim_{n\to\infty}\w_{f^n}(t)=0$ for all $t>0$;
\item {\em Edelstein contracting} if $d(f(x),f(y))<d(x,y)$ for any distinct points $x,y\in X$.
\end{itemize}
These and many other contracting conditions (expect for the eventual contractivity, which seems to be new)
are discussed in \cite{Jah} where the following implications are proved (except for ``Matkowski $\Ra$ Eventual'' which follows from $\w_{f^n}\leq\w_f^n$):
\smallskip

\centerline{Banach $\Ra$ Rakotch $\Ra$ Krasnoselski\u\i{} $\Ra$ Matkowski $\Ra$ Edelstein \&\ Eventual.}
\medskip

Moreover, for a self-map $f:X\to X$ of a bounded (compact) metric space $X$ the Krasnoselski\u\i{} (Edelstein) contractivity is equivalent to the Rakotch contractivity.

A standard application of the Banach Contraction Principle shows that each function system $\F\subset X^X$ consisting of Banach contracting maps on a complete metric space is topologically contracting and has an attractor. The same conclusion holds for function systems consisting of Matkowski or eventual contractions. This can be proved using the Matkowski or Leader Fixed Point Theorems \cite{Mat} or \cite{Led83}. The example of the Edelstein contracting map $f:[0,\infty)\to[0,\infty)$, $f(x)=x+e^{-x}$, shows that this result cannot be further generalized to function systems consisting of Edelstein contracting maps. On the other hand, in Theorem~\ref{t4.7} we shall show that a necessary and sufficient condition for a Edelstein contracting function system $\F$ to have an attractor is the compact-dominacy {(in fact, the known fixed point theorem due to Edelstein \cite{E} states that any Edelstein contraction on a compact metric space has an attracting fixed point)}.

The present investigation was motivated by the  problem of detecting topologically contracting function systems $\F$ on a topological space $X$ which are Banach (or Rakotch, Krasnoselski\u\i{}, Matkowski, Edelstein) contracting with respect to a metric (or a family of pseudometrics) generating the topology of $X$.
The ``Banach contracting'' case of this problem was considered in \cite{K}. One of our principal results is Theorem~\ref{t9.2} saying that each topologically contracting function system $\F$ on a (completely) metrizable space $X$ is Edelstein contracting with respect to some (complete) metric $d$ generating
the topology of $X$. Another main result is Theorem~\ref{t6.12n} saying that a function system $\F$ on a topological space $X$ is Krasnoselski\u\i{} contracting with respect to a suitable admissible (complete) metric if and only if $\F$ is Matkowski contracting with respect to a suitable admissible (complete) metric on $X$ if and only if $\F$ is eventually contracting with respect to a suitable admissible (complete) metric on $X$.

In fact, these metrization theorems can be generalized also to topologically contracting function systems on non-metrizable topological spaces. This will be done with help of the notion of a multimetric space, which is a set $X$ endowed with the family of pseudometrics $\mathcal D$ generating a Hausdorff topology on $X$. Multimetric spaces and their hyperspaces will be considered in Section~\ref{s:multimet}. In Sections~\ref{s3} and \ref{s4} we discuss various contractivity conditions for functions and function systems on multimetric spaces. In particular, in Section~\ref{s4} we shall prove one of the principal results of this paper saying that each eventually contracting function system on a sequentially complete multimetric space $(X,\mathcal D)$ is topologically contracting and hence has an attractor. From this theorem we shall derive that each compact-dominating Edelstein contracting function system is topologically contracting and has an attractor. The last Section~\ref{s9} contains metrization results for various sorts of contractive functions systems.


\section{Contracting function systems on topological spaces}

In this section we discuss various sorts of topological contractive properties of function systems on topological spaces. Given two families $\U,\V$ of sets, we shall write $\U\prec\V$ if each set $U\in\U$ is contained in some set $V\in\V$.
{
\begin{definition} Let $\F\subset X^X$ be a function system on a topological space $X$ such that $\F(C)\subset C$ for some non-empty compact subset $C\subset X$. The function system $\F$ is called
\begin{enumerate}
\item {\em compactly contracting} if for every compact subset $K\subset X$ and every open cover $\U$ on $X$ there is $n\in\w$ such that $\{f(K):f\in\F^{\ge n}\}\prec\U$;
\item {\em locally contracting} if for every compact subset $K\subset X$ and every open cover $\U$ on $X$ there are a neighborhood $O_K\subset X$ of the set $K$, and a number $n\in\w$ such that $\{f(O_K):f\in\F^{\ge n}\}\prec\U$;
\item {\em totally contracting} if for every compact subset $K\subset X$ there is a neighborhood $O_K\subset X$ of the set $K$, such that for every open cover $\U$ on $X$ there is a number $n\in\w$ such that $\{f(O_K):f\in\F^{\ge n}\}\prec\U$;
\item {\em globally contracting} if for every open cover $\U$ on $X$ there is a number $n\in\w$ such that $\{f(X):f\in\F^{\ge n}\}\prec\U$.
\end{enumerate}
In the sequel saying that some function system $\F\subset X^X$ is compactly (resp. locally, totally, globally) contracting, we shall always assume that $\F(C)\subset C$ for some nonempty compact subset $C\subset X$.
\end{definition}
}
These notions relate as follows:
\smallskip

\centerline{globally contracting $\Ra$ totally contracting $\Ra$ locally contracting $\Ra$ compactly contracting.}
\medskip

Moreover, for a function system $\F\subset X^X$ on a compact Hausdorff space $X$ all these contractivity properties are equivalent.

It turns out that the topological contractivity of function systems is equivalent to the compact contractivity. By $cl_X(A)$ we denote the closure of a set $A$ in a topological space $X$.

\begin{theorem}\label{t7.3} For a function system $\F\subset X^X$ on a Hausdorff space $X$, the following conditions are equivalent:
\begin{enumerate}
\item $\F$ is topologically contracting;
\item $\F$ is compactly contracting;
\item for every set $K\in \K(X)$ and every open cover $\U$ of $X$ there is $k\in\IN$ such that $$\{cl_X(f\circ\F^{<\omega}(K)):f\in\F^k\}\prec\U.$$
\end{enumerate}
Moreover, if $X$ is a regular space, then \textup{(1)--(3)} are equivalent to:
\begin{itemize}
\item[$(4)$] for every set $K\in \K(X)$ and every open cover $\U$ of $X$ there is $k\in\IN$ such that $$\{f\circ\F^{<\omega}(K):f\in\F^k\}\prec\U.$$
\end{itemize}
\end{theorem}

\begin{proof} To prove that $(1)\Ra(3)$, assume that $\F$ is topologically contracting and fix a compact set $K\in\K(X)$ and an open cover $\U$ of $X$. Since $\F$ is compact-dominating, there is a compact set $D\in\K(X)$ such that $K\subset D$ and $\F(D)\subset D$. By the topological contractivity of $\F$, for every sequence $\vec f=(f_n)_{n\in\w}\in\F^\w$ the intersection $\bigcap_{n\in\w}f_0\circ \dots\circ f_n(D)$ is a singleton contained in some open set $U\in\U$. By the compactness of $D$ there is a number $n=n(\vec f)$ such that $f_0\circ\dots \circ f_{n-1}(D)\subset U$. It follows that the set $V_{\vec f}=\{(g_n)_{n\in\w}\in\F^\w:\forall n<n(\vec f)\;\;g_n=f_n\}$ is an open neighborhood of $\vec f$ in the Tychonoff product topology on $\F^\w$. By the compactness of $\F^\w$, the cover $\{V_{\vec f}:\vec f\in\F^\w\}$ has a finite subcover $\{V_{\vec f}:\vec f\in E\}$ for some finite set $E\subset \F^\w$. Then for the number $n=\max_{\vec f\in E}n(\vec f)$ and each $f\in\F^n$ the set $f(D)$ is contained in some set $U\in\U$. Consequently,
$cl_X(f\circ \F^{<\omega}(K))\subset  cl_X(f(D))=f(D)\subset U.$
\smallskip

$(3)\Ra(1)$ Given any compact set $K\in\K(X)$, we shall prove that the closed subset
$$D:=cl_X(\F^{<\omega}(K))$$ of $X$ is compact. Given an open cover $\U$ of $X$, we need to find a finite subcover of $D$. The condition (3) yields a number $k\in\IN$ such that for every $f\in\F^k$ the set $\cl_X(f\circ \F^{<\w}(K)$) is contained in some set $U_f\in\U$. By the compactness of the set $\F^{<k}(K)=\bigcup_{n<k}\F^n(K)$, there is a finite subcover $\V$ of $\U$ such that $\F^{<k}(K)\subset\bigcup\V$. Then $\V\cup\{U_f:f\in\F^k\}\subset\U$ is a finite cover of the set $D=\cl_X(\F^{<\w}(K))=\F^{<k}(K)\cup\bigcup_{f\in \F^k}\cl_X(f\circ\F^{<\w}(K))$, witnessing that the set $D$ is  compact. Since $K\subset D$ and $\F(D)\subset D$, the function system $\F$ is compact-dominating.

Now take any $D\in\K(X)$ with $\F(D)\subset D$ and a function sequence $(f_n)_{n\in\w}\subset \F^\w$.
By the compactness of $D$ the decreasing sequence $(f_0\circ\dots\circ f_n(D))_{n\in\w}$ has non-empty intersection $K=\bigcap_{n\in\w}f_0\circ\dots\circ f_n(D)$. We need to show that $K$ is a singleton. Assume conversely that $K$ contains two distinct points $x,y$. Since $X$ is Hausdorff, these points have disjoint open neighborhoods $U_x,U_y$ in $X$. By the condition (3), for the open cover $\U=\{U_x,U_y,X\setminus\{x,y\}\}$ there is a number $n\in\w$ such that $\{x,y\}\subset f_0\circ\dots\circ f_n(D)\subset U$ for some $U\in\U$, which contradicts the choice of $\U$.
\smallskip

The equivalent conditions $(1)$ and $(3)$ trivially imply $(2)$.
\smallskip

$(2)\Ra(1)$ First we prove that the function system $\F$ is compact dominating. Take any $K\in\K(X)$. By $(2)$, there is a non-empty compact set $C\subset X$ with $\F(C)\subset C$. Replacing $K$ by $K\cup C$, we can assume that $C\subset K$. Our goal is to prove that the set $\F^{<\w}(K)$ is compact.
Let $\U$ be a cover of $\F^{<\w}(K)$ by open subsets of $X$. Using the compactness of the set $C$, find a finite subcover $\U_C\subset\U$ of $C$. By (3) for the open cover $\tilde\U_C=\U_C\cup\{X\setminus C\}$ of $X$ there is $k\in\w$ such that $\{f(K):f\in\F^{\ge k}\}\prec\tilde\U_C$.
Observe that for every $f\in\F^{\ge k}$ the set $f(K)$ contains the subset $f(C)\subset C$ and hence $f(K)\not\subset X\setminus C\in\widetilde{\U}_C$. This implies that $\U_C$ is a finite cover of the set $\F^{\ge k}(K)$. Since the set $\F^{<k}(K)=\bigcup_{n<k}\F^n(K)$ is compact, we can find a finite subcover $\V\subset\U$ of $\F^{<k}(K)$. Then $\V\cup\U_C\subset \U$ is a finite cover of $\F^{<\w}(K)$, witnessing that the set $\F^{<\w}(K)$ is compact. So, $\F$ is compact-dominating. Repeating the argument from the proof of the implication $(3)\Ra(1)$, we can prove that for every non-empty compact set $K\subset X$ with $\F(K)\subset K$ and every function sequence $(f_n)_{n\in\w}\in\F^\w$ the intersection $\bigcap_{n\in\w}f_0\circ\dots\circ f_n(K)$ is a singleton.
\smallskip

The implication $(3)\Ra(4)$ is trivial. The converse implication holds if the space $X$ is regular. In this case for every open cover $\U$ of $X$ we can find an open cover $\V$ of $X$ such that $\{\cl_X(V):V\in\V\}\prec\U$. By (4), for every compact set $K\in\K(X)$ there is $n\in\w$ such that $\{f\circ\F^{<\w}(K):f\in\F^n\}\prec\V$. Then $$\{\cl_X(f\circ\F^{<\w}(K)):f\in\F^n\}\prec\{\cl_X(V):V\in\V\}\prec\U$$witnessing the condition (3).
\end{proof}

Next, we show that the topological contractivity of function systems on $k$-spaces is equivalent to the  local contractivity. Let us recall that a topological space $X$ is called a {\em $k$-space} if its topology is determined by compact sets in the sense that a subset $U\subset X$ is open if and only if for every compact subset $K\subset X$ the intersection $K\cap U$ is open in the subspace topology on $K$. It is well-known \cite[\S3.3]{Eng} that the class of $k$-spaces includes all locally compact spaces and all first countable spaces.

\begin{theorem}\label{t5.3} A function system $\F$ on a Hausdorff $k$-space $X$ is topologically contracting if and only if it is locally contracting.
\end{theorem}

\begin{proof}
Assume that $\F$ is topologically contracting. By Theorem~\ref{Mih}, the function system $\F$ has an attractor $A_\F$. Given any compact set $K\in\K(X)$ and open cover $\U$ of $X$, we need to find $n\in\w$ and a neighborhood $O_K\subset X$ of $K$ such that $\{f(O_K):f\in\F^{\ge n}\}\prec\U$. Replacing $K$ by a larger compact set we can assume that $A_\F\subset K$ and $\F(K)\subset K$. By Theorem~\ref{t7.3}, there is $n\in \IN$ such that for every $f\in \F^n$, there is $U_f\in\U$ such that for every $g\in\F^{<\omega}$, $f\circ g(K)\subset U_f$ (and, in particular, $f(A_\F)\subset U_f$).\\
It is clear that the set
$$
O_K=\bigcap_{f\in\F^n}\bigcap_{g\in\F^{<\omega}}(f\circ g)^{-1}(U_f)
$$
contains $K$ and $f\circ g(O_K)\subset U_f\in\U$ for every $f\in\F^n$ and $g\in\F^{<\w}$.
It remains to show that the set $O_K$ is open in $X$ or equivalently $X\setminus O_K$ is closed in $X$. In the opposite case we can find a compact set $C\subset X$ such that $C\setminus O_K$ is not closed in the $k$-space $X$. Since for every compact set $D\subset X$ containing $C$ we get $(D\setminus O_K)\cap C=C\setminus O_K$, we can assume that $K\subset C$ and $\F(C)\subset C$.

Since $C\setminus O_K$ is not closed (both in $X$ and $C$), there is a point $y\in cl_C(C\setminus O_K)\setminus (C\setminus O_K)$. This point belongs to $y\in C\cap O_K$ and each its neighborhood $O_y$ (in $C$) meets $C\setminus O_K$.

For every $x\in A_\F$ consider the finite set $\F_x=\{f\in\F^n:x\in f(A_\F)\}$ and the open neighborhood
$$
V_x=\bigcap_{f\in\F_x}U_f\setminus\bigcup_{f\in\F^n\setminus\F_x}f(A_\F)
$$of the point $x$ in $X$.
Then $\mV=\{V_x:x\in A_\F\}\cup\{X\setminus A_\F\}$ is an open cover of $X$.
Theorem~\ref{t7.3} yields a number $m\geq n$ such that for every $h\in \F^{\ge m}$ the set $h(C)$ is contained in some set $V_h\in\V$.

Since the family $\{f\circ g:f\in\F^n,\;g\in\F^{\le m-n}\}$ is finite and consists of continuous functions,  we can choose a relatively  open neighborhood $O_y\subset C$ of $y$ in $C$ such that $f\circ g(O_y)\subset U_f$ for all $f\in\F^n$ and $g\in\F^{\le m-n}$ (here we use the fact that $f\circ g(y)\in U_f$, which follows from $y\in O_K$). On the other hand, for every $g\in\F^{>m-n}$ the choice of the number $m$ guarantees that { for every $f\in\F^n$}, $f\circ g(O_y)\subset f\circ g(C)\subset V_{f\circ g}\in\V$. Since $f\circ g(A_\F)\subset f\circ g(C)\cap A_\F$, the set $f\circ g(C)$ is not contained in the set $X\setminus A_\F\in\V$ and hence $V_{f\circ g}=V_x$ for some $x\in A_\F$. Taking into account that the intersection $f(A_\F)\cap V_x\supset f\circ g(A_\F)$ is not empty, we conclude that $f\in\F_x$ and hence $f\circ g(O_y)\subset f\circ g(C)\subset V_x\subset U_f$. Therefore, we have shown that for all $f\in\F^n$ and $g\in\F^{<\w}$, $f\circ g(O_y)\subset U_f$, which means that $O_y\subset O_K$ and $y$ is an interior point of the set $O_K\cap C$. But this contradicts the choice of the point $y\in \cl_C(C\setminus O_K)$.
\end{proof}

\section{Multimetric spaces and their hyperspaces}\label{s:multimet}

In this section we introduce the concept of a multimetric space and shall consider hyperspaces of such spaces.

Any family of pseudometrics $\mathcal D$ on a set $X$ will be called a {\em multipseudometric} on $X$.
A multipseudometric $\mathcal D$ on $X$ is called a {\em multimetric} if for any distinct points $x,y\in X$ there is a pseudometric $d\in\mathcal D$ such that $d(x,y)>0$.

{ Note that families of pseudometrics were deeply investigated in the literature. In particular, uniform structures can be equivalently defined via multimetrics. However, for our purposes it is more fruitful to work with multimetrics rather then uniform structures (which ``forgot'' some structure information).}

For a point $x$ of a set $X$, a real number $\e>0$, and a pseudometric $d$ on $X$ by $B_d(x,\e)=\{y\in X:d(y,x)<\e\}$ we shall denote the open $d$-ball of radius $\e$ centered at $x$. Moreover, for a finite family $\mathcal D$ of pseudometrics on $X$ we denote by $B_{\mathcal D}(x,\e)=\bigcap_{d\in\mathcal D}B_d(x,\e)$ the $\mathcal D$-ball of radius $\e$ centered at $x$. It is clear that $B_{\mathcal D}(x,\e)=B_{d}(x,\e)$ for the pseudometric $\bar d=\max\mathcal D$. For a subset $A\subset X$ and a pseudometric $d$ on $X$ let $B_d(A,\e)=\bigcup_{a\in A}B_d(a,\e)$ be the $\e$-neighborhood of $A$.

A multipseudometric $\mathcal D$ is called:
\begin{itemize}
\item {\em bounded} if each pseudometric $d\in\mathcal D$ is bounded in the sense that $\diam_d(X):=\sup_{x,y\in X}d(x,y)<\infty$;
\item {\em totally bounded} if each pseudometric $d\in\mathcal D$ is totally bounded in the sense that for every $\e>0$ there is a finite subset $F\subset X$ such that $X=\bigcup_{x\in F}B_d(x,\e)$.
\end{itemize}

By a {\em multi}({\em pseudo}){\em metric space} we shall understand a pair $(X,\mathcal D)$ consisting of a set $X$ and a multi(pseudo)metric $\mathcal D$ on $X$.
Any multipseudometric $\mathcal D$ on a set $X$ generates the {\em canonical topology} on $X$ consisting of sets $U\subset X$ such that for each $x\in U$ there is $\e>0$ and a finite subfamily $\E\subset\mathcal D$ such that $B_\E(x,\e)\subset U$. Observe that the canonical topology is the smallest topology on $X$ in which all pseudometrics $d\in\mathcal D$ are continuous. The canonical topology on a multipseudometric space $(X,\mathcal D)$ is Hausdorff if and only if the multipseudometric $\mathcal D$ is a multimetric. In this case the canonical topology is Tychonoff (since it is generated by the uniform structure generated by the multimetric).
Conversely, the topology of each Tychonoff space is generated by a suitable multimetric (see, e.g. \cite[\S 8.1]{Eng}).

A multimetric $\mathcal D$ on a topological space $X$ is called {\em admissible} if it generates the topology of $X$.
The cardinal
$$\mn(X)=\min\{|\mathcal D|:\mbox{$\mathcal D$ is an admissible multimetric on $X$}\}$$
is called the {\em metrizability number} of a Tychonoff space. It is equal to the smallest cardinality $|\mathcal M|$ of a family $\mathcal M$ of metric spaces whose Tychonoff product $\prod\mathcal M$ contains a subspace homeomorphic to $X$. It is clear that a topological space $X$ is metrizable if and only if $\mn(X)=1$ if and only if $\mn(X)\le\aleph_0$.

Also we shall need the cardinal
$$\tmn(X)=\min\{|\mathcal D|:\mbox{$\mathcal D$ is an admissible totally bounded multimetric on $X$}\},$$
which is equal to the smallest cardinality $|\mathcal M|$ of a family $\mathcal M$ of totally bounded metric spaces whose Tychonoff product $\prod\mathcal M$ contains a subspace homeomorphic to $X$. It is easy to see { (for example, by considering embeddings to a Tychonoff cube) } that
$$\tmn(X)=\begin{cases}1&\mbox{if $w(X)\le\aleph_0$},\\
w(X)&\mbox{if $w(X)>\aleph_0$},
\end{cases}
$$where $w(X)$ is the {\em weight} of $X$, i.e., the smallest cardinality of a base of the topology of $X$.

Each topological space $X$ carries the {\em universal multipseudometric} $\mathcal D_X$ consistsing of all continuous pseudometrics on $X$. If the space $X$ is Tychonoff, then its universal multipseudometric $\mathcal D_X$ is an admissible multimetric on $X$. This allows us to speak about multimetric properties of Tychonoff spaces (such as the sequential completeness discussed later).

For a multimetric space $(X,\mathcal D)$ by $\K(X)$ we denote the space of all non-empty compact subsets of $X$, endowed with the Vietoris topology. The hyperspace $\K(X)$ carries the induced multipseudometric $$\mathcal D_H=\{d_H:d\in\mathcal D\}$$consisting of the Hausdorff pseudometrics
$$d_H(A,B)=\max\big\{\max_{a\in A}d(a,B),\max_{b\in B}d(b,A)\big\}.$$
In general, the family $\mathcal D_H$ does not generate the Vietoris topology on the hyperspace $\K(X)$, so is not admissible.

\begin{example} For the multimetric space $(\IR^2,\mathcal D)$ where $\mathcal D=\{d_1,d_2\}$ and
$$d_1\big((x,y),(x',y')\big)=|x-x'|\mbox{ \ and \ }d_2\big((x,y),(x',y')\big)=|y-y'|,$$
the family $\mathcal D_H$ does not generate the Vietoris topology on $\K(X)$.
\end{example}

\begin{proof} The topology generated by the family $\mathcal D_H$ on $\K(X)$ is not Hausdorff as it does not distinguish the compact sets $[0,1]^2$ and $\Delta=\{(x,x):x\in[0,1]\}$.
\end{proof}

A multimetric $\mathcal D$ on a set $X$ will be called {\em directed} if for any two pseudometrics $d_1,d_2\in\mathcal D$ there is a pseudometric $d_3\in \mathcal D$ such that $d_3\ge\max\{d_1,d_2\}$.
{It is known (and easily seen)} that for each admissible multimetric $\mathcal D$ on a topological space $X$ the directed multimetric
$$\bar{\mathcal D}=\{\max\mathcal E:\mathcal E\subset\mathcal D,\;\;|\mathcal E|<\w\}$$generates the topology of $X$ and so is admissible.

\begin{proposition}\label{p2.1} If $\mathcal D$ is a directed multimetric on a set $X$, then $\mathcal D_H$ is an admissible multimetric on $\K(X)$.
\end{proposition}

\begin{proof} The continuity of each pseudometric $d\in\mathcal D$ with respect to the canonical topology generated by the family $\mathcal D$ on $X$ implies the continuity of the Hausdorff pseudometric $d_H$ on the hyperspace $\K(X)$. This means that the Vietoris topology on $\K(X)$ is stronger than the topology generated by the family $\mathcal D_H$.

It remains to show that the Vietoris topology on $\K(X)$ is weaker than the topology generated by $\mathcal D_H$. Fix any compact set $K\in\K(X)$ and a subbasic neighborhood $O_K\subset \K(X)$ of $K$, which is of the form $\langle U\rangle^+=\{C\in\K(X):C\subset U\}$ or $\langle U\rangle^-=\{C\in\K(X):C\cap U\neq\emptyset\}$ for some non-empty open set $U\subset X$.

If $O_K=\langle U\rangle^-$, then we can find a point $x\in K\cap U$ and choose a finite subfamily $\G\subset\mathcal D$  such that $B_\G(x,\e)\subset U$ for some $\e>0$. Since the family $\mathcal D$ is directed, there is a pseudometric $d\in\mathcal D$ such that $d\ge\max\G$. For this pseudometric we get $B_d(x,\e)\subset B_\G(x,\e)\subset U$ and $B_{d_H}(K,\e)=\{C\in\K(X):d_H(C,K)<\e\}\subset \langle U\rangle^-=O_K$.

Next, we consider the case $O_K=\langle U\rangle^+=\{C\in\K(X):C\subset U\}$. For every $x\in K$ choose a finite subfamily $\G_x\subset\mathcal D$ and $\e_x>0$ such that $B_{\G_x}(x,2\e_x)\subset U$. By the compactness of $K$ the open cover $\{B_{\G_x}(x,\e_x):x\in K\}$ of $K$ admits a finite subcover $\{B_{\G_{x}}(x,\e_{x}):x\in F\}$ where $F\subset K$ is a finite set. Let $\e=\min_{x\in F}\e_x$. Since the family $\mathcal D$ is directed, there is a pseudometric $d\in\mathcal D$ such that $d\ge\max_{x\in F}\G_x$. We claim that $B_{d_H}(K,\e)\subset \langle U\rangle^+$. Given any compact set $C\in B_{d_H}(K,\e)$, we need to show that $C\subset U$. Given any point $c\in C$, we can find a point $z\in K$ with $d(c,z)<\e$. For the point $z$ we can find a point $x\in F$ such that $z\in B_{\G_x}(x,\e_x)$. Then for every $\rho\in\G_x$ we get
$$\rho(x,c)\le\rho(x,z)+\rho(z,c)<\e_x+d(z,c)<\e_x+\e\le 2\e_x$$and hence $c\in B_{\G_x}(x,2\e_x)\subset U$. Therefore $B_{d_H}(K,\e)\subset \langle U\rangle^+=O_K$.
\end{proof}

We shall also need the following fact whose proof is standard and is left to the reader.

\begin{lemma}\label{l2.2} Let $(X,\mathcal D)$ be a multimetric space. For any open cover $\U$ of $X$ an any compact subset $K\subset X$ there are a finite subfamily $\G\subset\mathcal D$ and $\e>0$ such that for each $x\in X$ the ball $B_\G(x,\e)$ is contained in some set $U\in\U$.
\end{lemma}

Now we shall discuss the notion of sequential completeness of  multimetric and topological spaces.

Let $(X,\mathcal D)$ be a multimetric space. A sequence $(x_n)_{n=1}^\infty$ of points of $X$ will be called {\em Cauchy} if it is $d$-Cauchy for every pseudometric $d\in\mathcal D$. The latter means that for every $\e>0$ there is $n\in\IN$ such that $d(x_m,x_k)<\e$ for every $m,k\ge n$.

A multimetric space $(X,\mathcal D)$ is called {\em sequentially complete} if each Cauchy sequence $(x_n)_{n=1}^\infty$ in $X$ converges to some point $x_\infty$. The point $x_\infty$ is unique since the canonical topology on $X$ is Hausdorff.

\begin{proposition}\label{ff2}
Each compact multimetric space $(X,\D)$ is sequentially complete.
\end{proposition}

\begin{proof}
Let $(x_n)$ be a Cauchy sequence in $X$. By the compactness of $X$, the sequence $(x_n)$ has an accumulation point $x_\infty\in X$. To show that $(x_n)$ converges to $x_\infty$, take any finite family $\G\subset \mathcal D$ and $\e>0$. Since $x_\infty$ is accumulation point of $(x_n)$, there is an increasing number sequence $(n_k)$ such that $(x_{n_k})\subset B_\G(x_\infty,\e/2)$. Since $(x_n)$ is Cauchy, there is $k_0$ such that $d(x_m,x_n)<\e/2$ for every $m,n\geq k_0$ and $d\in\G$. Then for any $k\geq k_0$ and $d\in\G$, we have
$$
d(x_k,x_\infty)\leq d(x_k,x_{n_k})+d(x_{n_k},x_\infty)<\frac\e2+\frac\e2=\e
$$
which means that $x_k\in B_\G(x_\infty,\e)$ and the result follows.
\end{proof}

A Tychonoff space $X$ is called {\em sequentially complete} if it is sequentially complete with respect to its universal multimetric $\mathcal D_X$ (consisting of all continuous pseudometrics on $X$). The class of sequentially complete Tychonoff spaces is quite wide.

\begin{proposition} Each normal topological space $X$ is sequentially complete.
\end{proposition}

\begin{proof} We need to prove that each $\mathcal D_X$-Cauchy sequence $(x_n)_{n\in\w}$ in $X$ converges, equivalently, has an accumulation point $x_\infty\in X$. This is trivially true if the set $\{x_n\}_{n\in\w}$ is finite. If this set is infinite, then we can choose a subsequence $(x_{n_k})_{k\in\w}$ consisting of pairwise distinct points. If this subsequence has no accumulating points in $X$, then the sets $A=\{x_{n_{2k}}\}_{k\in\w}$ and $B=\{x_{n_{2k+1}}\}_{k\in\w}$ are closed disjoint subsets of $X$. By Urysohn's Lemma, there is a continuous function $f:X\to [0,1]$ such that $f(A)\subset\{0\}$ and $f(B)\subset \{1\}$. The function $f$ induces a continuous pseudometric $d(x,y)=|f(x)-f(y)|$ on $X$ with respect to which the sequence $(x_{n_k})_{k\in\w}$ is not Cauchy. This contradiction shows that the $\mathcal D_X$-Cauchy sequence $(x_n)_{n\in\w}$ converges.
\end{proof}

A topological space $X$ is called {\em Dieudonn\'e complete} if $X$ is homeomorphic to a closed subspace of a Tychonoff product of complete metric spaces; see \cite[8.5.13]{Eng}. A subset $A$ of a topological space $X$ is called {\em seqeuntially closed} in $X$ if $A$ contains the limits of all sequences $\{a_n\}_{n\in\w}\subset A$ that converge in $X$.

\begin{proposition} A Tychonoff space $X$ is sequentially complete if and only if $X$ is sequentially closed in some Dieudonn\'e complete topological space.
\end{proposition}

\begin{proof} Assume that $X$ is sequentially closed in some Dieudonn\'e complete topological space $Y$. By definition, $Y$ can be identified with a closed subspace of a Tychonoff product $\prod_{\alpha\in A}M_\alpha$ of some complete metric spaces $(M_\alpha,d_\alpha)$, $\alpha\in A$. For every $\alpha\in A$ consider the projection $\pr_\alpha:Y\to M_\alpha$ onto the $\alpha$-factor and the induced pseudometric $\tilde d_\alpha(x,y))=d_\alpha(\pr_\alpha(x),\pr_\alpha(y))$ on $Y$.

To show that $X$ is sequentially complete, fix any $\mathcal D_X$-Cauchy sequence $(x_n)_{n\in\w}$ in $X$. Since each pseudometric $\tilde d_\alpha|X\times X$, $\alpha\in A$, is continuous, the sequence $(x_n)_{n\in\w}$ is Cauchy with respect to the pseudometric $\tilde d_\alpha$ and the sequence $(\pr_\alpha(x_n))_{n\in\w}$ is Cauchy in the complete metric space $M_\alpha$. By the completeness, the latter sequence converges to some point $y_\alpha\in M_\alpha$.
It follows that the point $y=(y_\alpha)_{\alpha\in A}\in\prod_{\alpha\in A}M_\alpha$ is the limit of the sequence $(x_n)_{n\in\w}$ and hence belongs to the closed subset $Y$ of $\prod_{\alpha\in A}M_\alpha$ as well as to the sequentially closed subset $X$ of $Y$. Therefore the sequence $(x_n)_{n\in\w}$ converges to the point $y\in X$ and the space $X$ is sequentially complete.
\smallskip

Now assume that the Tychonoff space $X$ is sequentially complete. For each continuous pseudometric $d\in \mathcal D_X$ on $X$ consider the quotient space $X_d=X/\kern-2pt\sim_d$ of $X$ by the equivalence relation $\sim_d=\{(x,y)\in X\times X:d(x,y)=0\}$. Let $q_d:X\to X_d$ be the quotient map assigning to each point $x\in X$ its equivalence class $[x]=\{y\in X:d(x,y)=0\}$. The pseudometric $d$ induces the metric $\tilde d$ on $X_d$ such that $\tilde d([x],[y])=d(x,y)$ for any points $x,y\in X$. Let $(\hat X_d,\hat d)$ be the completion of the metric space $(X_d,\tilde d)$. Consider the diagonal embedding $e:X\to\prod_{d\in\mathcal D_X}\hat X_d$, $e:x\mapsto (q_d(x))_{d\in\mathcal D_X}$. By definition, the closure $Y$ of $e(X)$ in $\prod_{d\in\mathcal D_X}\hat X_d$ is Dieudonn\'e complete. Since $e:X\to Y$ is a topological embedding, we can identify $X$ with its image $e(X)$ in $Y$. It remains to check that $e(X)$ is sequentially closed in $Y$. {Let $(x_n)_{n\in\w}\in e(X)^\w$ be a sequence that converges to some point $y\in Y$, and let $(z_n)_{n\in\w}\in X^\omega$ be such that $x_n=e(z_n)$ for all $n\in\w$. Then for every pseudometric $d\in\mathcal D_X$ the sequence $\big(\pr_d(x_n)\big)_{n\in\w}$  converges to $\pr_d(y)$ in the complete metric space $\hat X_d$ ($\pr_d:Y\to \hat X_d$ is the projection). Consequently, $\big(q_d(z_n)\big)_{n\in\w}=\big(\pr_d(x_n)\big)_{n\in\w}$ is Cauchy with respect to the
metric $\hat d$ and $(z_n)_{n\in\w}$ is Cauchy with respect to the pseudometric $d$. Therefore, the sequence $(z_n)_{n\in\w}$ is $\mathcal D_X$-Cauchy and by the sequential completeness, it converges to some point $z$ of $X$. Then $(x_n)_{n\in\N}$ converges to $e(z)\in e(X)$, and $e(z)=y$ (as the space $Y$ is Hausdorff). So, we have proved that the limit point $y=\lim_{n\to\infty} x_n$ belongs to $e(X)$, which means that $e(X)$} is sequentially closed in $Y$.
\end{proof}

\begin{corollary} Each Dieudonn\'e complete topological space is sequentially complete.
\end{corollary}

\begin{example} For any uncountable cardinal $\kappa$ the space $X=[0,1]^\kappa\setminus\{0\}^\kappa$ is not sequentially complete.
\end{example}

\begin{proof}
It suffices to prove that any sequence $(x_n)_{n\in\w}\in X^\w$ that converges to the unique point $\mathbf 0\in [0,1]^\kappa\setminus X$ is Cauchy with respect to any continuous pseudometric $d$ on $X$. Let $X_d=X/\sim_d$ be the quotient space of $X$ by the equivalence relation $\sim_d$ defined by $x\sim_d y$ iff $d(x,y)=0$. On the space $X_d$ the pseudometric $d$ induces a metric $\hat d$ such that the quotient map $q_d:X\to X_d$ is an isometry. Since the space $X=[0,1]^\kappa\setminus\{0\}^\kappa$ is pseudocompact (which follows from \cite[3.10.17]{Eng}), its continuous metric image $X_d$ is compact. The Glicksberg Theorem \cite[3.12.24(g)]{Eng} implies that the Tychonoff cube $[0,1]^\kappa$ is a Stone-\v Cech compactification of $X$, which allows us to extend the quotient map $q_d:X\to X_d$ to a continuous map $\bar q_d:[0,1]^\kappa\to X_d$.
Then the formula $\bar d(x,y)=\hat d(\bar q_d(x),\bar q_d(y))$ determines a continuous pseudometric $\bar d$ on the Tychonoff cube $[0,1]^\kappa$, which extends the pseudometric $d$. The sequence $(x_n)_{n\in\w}$ converges to zero in $[0,1]^\kappa$ and hence is Cauchy with respect to the pseudometric $\bar d$ and to its restriction $d$.
\end{proof}

It is well-known \cite[4.5.23]{Eng} that for every complete metric space $(X,d)$ its hyperspace $\K(X)$ is complete with respect to the Hausdorff metric $d_H$. This fact cannot be generalized to sequentially complete multimetric spaces.

\begin{example}\label{ex2.2} There is a sequentially complete multimetric space $(X,\mathcal D)$ endowed with a directed family of pseudometrics $\mathcal D$ such that the multimetric space $(\K(X),\mathcal D_H)$ is not sequentially complete.
\end{example}

\begin{proof} Let $X$ be the Banach space $\ell_1=\{(x_n)_{n\in\w}\in\IR^\w:\sum_{n=0}^\infty|x_n|<\infty\}$ endowed with the directed family $\mathcal D=\{d_{F}:F\subset X^*,\;|F|<\infty\}$ consisting of the pseudometrics
 $$d_{F}(x,y)=\max_{x^*\in F}|x^*(x)-x^*(y)|,\;\;x,y\in X,$$
 indexed by finite subsets $F$ of the dual Banach space $X^*=\ell_\infty:=\{(x_n)_{n\in\w}\in\IR^\w:\sup_{n\in\w}|x_n|<\infty\}$.
The family $\mathcal D$ generates the weak topology on the space $X$. It is known \cite[p.91]{Dis} that the Banach space $\ell_1$ is sequentially complete in the weak topology, which means that the multimetric space $(X,\mathcal D)$ is  sequentially complete.

Since the Banach space $X=\ell_1$ is not reflexive, the closed unit ball $B=\{(x_n)_{n\in\w}\in\ell_1:\sum_{n=0}^\infty|x_n|\le 1\}$ is not compact in the weak topology of $X$.
So, it does not belong to the hyperspace $\K(X)$. Since $B$ is separable, we can choose an increasing sequence $(K_n)_{n\in\w}$ of finite subsets of $B$ whose union $\bigcup_{n\in\w}K_n$ is dense in $B$.
It follows that $(K_n)_{n\in\w}$ is a Cauchy sequence in the multimetric space $(\K(X),\mathcal D_H)$ which has no limit in the hyperspace $\K(X)$ (as $B\notin\K(X)$ and $d_H(K_n,B)\rightarrow 0$ for every $d\in\mathcal D$).
\end{proof}

\section{Contractive maps on multimetric spaces}\label{s3}

In this section we shall discuss some contractivity properties of self-maps of multimetric spaces and establish a Fixed Point Theorem~\ref{t3.4} for such maps. In fact, for self-maps of complete metric spaces, Theorem~\ref{t3.4} can be derived from known Fixed Point Theorems, in particular, that of Leader \cite{Led83}.

It will be convenient to define contracting conditions for {\em partial self-maps} of spaces $X$, i.e., functions $f:\dom(f)\to X$ defined on subsets $\dom(f)\subset X$. For two partial maps $f:\dom(f)\to X$ and $g:\dom(g)\to X$ their composition $f\circ g$ is the partial map defined on the set $\dom(f\circ g)=g^{-1}(\dom(f))$ by the formula $f\circ g(x)=f(g(x))$ for $x\in\dom(f\circ g)$.
For a partial self-map $f:\dom(f)\to X$ by $f^n$ we denote its $n$-th iteration.

By $X^{\le X}$ we denote the set of all partial self-maps of $X$. It is a semigroup with respect to the operation of composition of partial functions. The identity map $\id_X$ of $X$ is the (two-sided) unit of the semigroup $X^{\le X}$. The semigroup $X^{\le X}$ contains the semigroup of all self-maps of $X$ as a subsemigroup. For a subset $\F\subset X^{\le X}$ we put $\F^0=\{\id_X\}$ and $\F^{n+1}=\F^n\circ\F=\{f\circ g:f\in\F^n,\;g\in\F\}$ for $n\in\w$.

Let $d$ be a pseudometric on a space $X$ and $f:\dom(f)\to X$ be a partial self-map of $X$. The function
$$d\w_f:[0,\infty]\to[0,\infty],\;\;d\w_f:t\mapsto\sup\big(\{0\}\cup\{d(f(x),f(y)):x,y\in \dom(f),\;d(x,y)\le t\}\big)$$will be called the {\em $d$-oscillation} of $f$. For $n\in\IN$ by $d\w_f^n$ we denote the $n$th iteration of the function $d\w_f$ and by {$d\w_{f^n}$} the oscillation of the $n$-th iteration of $f$. It is clear that {$d\w_{f^n}\le d\w_f^n$}.

\begin{definition} Let $(X,\mathcal D)$ be a multipseudometric space. A partial self-map $f:\dom(f)\to X$ of $X$ is called
\begin{itemize}
\item {\em Banach  contracting} if $\sup_{0<t<\infty}{d\w_f(t)}/{t}<1$ for every pseudometric $d\in\mathcal D$;
\item {\em Rakotch contracting} if $\sup_{a\le t<\infty}{d\w_f(t)}/{t}<1$ for every pseudometric $d\in\mathcal D$ and every number $a>0$;
\item {\em Krasnoselski\u\i{}  contracting} if $\sup_{a\le t\le b}{d\w_f(t)}/{t}<1$ for every pseudometric $d\in\mathcal D$ and every numbers $0<a<b<\infty$;
\item {\em Matkowski contracting} if  $\lim_{n\to\infty}d\w_f^n(t)=0$ for every $d\in\mathcal D$ and $t>0$;
\item {\em eventually contracting} if  $\lim_{n\to\infty}d\w_{f^n}(t)=0$ for every $d\in\mathcal D$ and $t>0$;
\item {\em Edelstein contracting} if for every $d\in\mathcal D$ and points $x,y\in \dom(f)$ there is a constant $\lambda<1$ such that $d(f(x),f(y))\le\lambda\cdot d(x,y)$.
\end{itemize}
\end{definition}

Proposition~\ref{p3.2} proved below implies that for any partial self-map $f$ on a multipseudometric space $(X,\mathcal D)$ we have the following implications between various contractivity conditions:
\smallskip

\centerline{Banach $\Ra$ Rakotch $\Ra$ Krasnoselski\u\i{} $\Ra$ Matkowski $\Ra$ Edelstein \&\ Eventual.}
\medskip

Moreover, if the domain $\dom(f)$ of $f$ is compact, then the Edelstein contractivity is equivalent to the Rakotch contractivity.

\begin{proposition}\label{p3.2} Let $(X,\mathcal D)$ be a multipseudometric space and $f:\dom(f)\to X$ be a partial self-map of $X$.
\begin{enumerate}
\item If $f$ is Banach contracting, then $f$ is Rakotch contracting;
\item If $f$ is Rakotch contracting, then $f$ is Krasnoselski\u\i{} contracting;
\item If $f$ is Krasnoselski\u\i{} contracting, then $f$ is Matkowski contracting;
\item If $f$ is Matkowski contracting, then $f$ is Edelstein and eventually contracting;
\item The map $f$ is Edelstein contracting if and only if for every compact subset $K\subset \dom(f)$ the restriction $f|K:K\to X$ is Rakotch contracting.
\end{enumerate}
\end{proposition}

\begin{proof} 1,2. The first two statements follow immediately from the definitions.
\smallskip

3. Assume that $f$ is Krasnoselski\u\i{} contracting. We need to check that $f$ is Matkowski contracting. Assuming the converse, we would find $b>0$ and a pseudometric $d\in\mathcal D$ such that $\delta=\lim_{n\to\infty}d\w^n_f(b)=\inf_{n\in\w}d\w^n_f(b)$ is strictly positive {(here we used the fact that $d\w_f(t)\le t$ for $t>0$)}. Since $f$ is Krasnoselski\u\i{} contracting, the number $\lambda=\sup_{\delta\le t\le b}\frac{d\w_f(t)}{t}<1$. Since $\delta=\lim_{n\to\infty}d\w^n_f(b)>0$, there is $n\in\w$ such that $d\w_f^{n+1}(b)>\lambda\cdot d\w^n_f(b)$. Then for the number $\tau=d\w^n_f(b)\in [\delta,b]$ we get $d\w_f(\tau)=d\w^{n+1}_f(b)>\lambda\cdot \tau$, which contradicts the definition  of $\lambda$. This contradiction shows that $f$ is Matkowski contracting.
\smallskip

4. Assume that $f$ is Matkowski contracting. We claim that $d\w_f(t)<t$ for every $t>0$ and every pseudometric $d\in\mathcal D$. Assuming conversely that $d\w_f(t)\ge t$ for some $t>0$ and using the monotonicity of $d\w_f$ we would conclude that $d\w^{n+1}_f(t)\ge d\w^n_f(t)$ for every $n\in\w$, which implies that $\lim_{n\to\infty}d\w^n_f(t)\ge t>0$. This is a desired contradiction showing that $d\w_f(t)<t$.

To show that $f$ is Edelstein contracting, take any points $x,y\in \dom(f)$ and any pseudometric $d\in\mathcal D$. If $d(x,y)=0$, then $d(f(x),f(y))\le d\w_f(d(x,y))\le d(x,y)=0$ and hence $0=d(f(x),f(y))\le\frac12 d(x,y)=0$. If $d(x,y)>0$, then $d(f(x),f(y))\le d\w_f(d(x,y))<d(x,y)$, which means that $f$ is Edelstein contracting.

The eventual contractivity of $f$ follows from the inequality $d\w_{f^n}\le d\w_f^n$.
\smallskip

5. The ``if'' part of the {fifith} statement follows immediately from the statements (2)--(4).
 To prove the ``only if'' part, assume that $f$ is Edelstein contracting. Given any compact subset $K\subset \dom(f)$, we need to check that $f|K$ is Rakotch contracting. Fix any pseudometric $d\in\mathcal D$ and observe that for every $\delta>0$ the set $K_\delta=\{(x,y)\in K\times K:d(x,y)\ge\delta\}$ is compact. The Edelstein contractivity of $f$ guarantees that $d(f(x),f(y))<d(x,y)$ for every pair $(x,y)\in K_\delta$. The continuity of the function $\frac{d(f(x),f(y))}{d(x,y)}$ on the compact space $K_\delta$ guarantees that
$$\sup_{\delta\le t<\infty}\frac{d\w_{f|K}(t)}{t}=\sup_{(x,y)\in K_\delta}\frac{d\big(f(x),f(y)\big)}{d(x,y)}=\max_{(x,y)\in K_\delta}\frac{d\big(f(x),f(y)\big)}{d(x,y)}<1,$$
which means that $f|K$ is Rakotch contracting.
\end{proof}

Eventually contracting maps on pseudometric spaces have the following nice property.

\begin{lemma}\label{l3.3} If $f:X\to X$ is an eventually contracting map on a pseudometric space $(X,d)$, then for every $x\in X$ the sequence $(f^n(x))_{n\in\w}$ is {$d$-}Cauchy.
\end{lemma}

\begin{proof} Assuming that $(f^n(x))_{n\in\w}$ is not Cauchy, we can find $\e>0$ and two number sequences $(n_k)_{k\in\w}$ and $(m_k)_{k\in\w}$ such that $m_k>n_k\ge k$ and $d(f^{m_k}(x),f^{n_k}(x))>3\e$ for all $k\in\w$.  Moreover, replacing $m_k$ by a smaller number, if necessary, we can assume that for every $k\in\N$, $d(f^{m_k}(x),f^{n_k}(x))>3\e$ and $d(f^{m}(x),f^{n_k}(x))\leq3\e$ for all $m\in [n_k,m_k)$.

Since $f$ is eventually contracting, there is a number $r\in\IN$ such that $d\w_{f^r}(4\e)<\e$. Put $$D=\max\{d(x,f(x)),d(x,f^r(x))\}$$ and using the eventual contractivity of $f$, find a number $\tilde r\ge r$ such that $d\w_{f^n}(D)<\e$ for all $n\ge\tilde r$. Choose $k\in\w$ so large that $m_k>n_k>k>\tilde r$ and observe that $d(f^{m_k}(x),f^{m_k-1}(x))\le d\w_{f^{m_k-1}}\big(d(f(x),x)\big)\le d\w_{f^{m_k-1}}(D)<\e$, which implies that
$d(f^{n_k}(x),f^{m_k}(x))\le d(f^{n_k}(x),f^{m_k-1}(x))+d(f^{m_k-1}(x),f^{m_k}(x))\le 3\e+\e=4\e$.
Then
$$
\begin{aligned}
3\e&<d\big(f^{n_k}(x),f^{m_k}(x)\big)\le d\big(f^{n_k}(x),f^{n_k+r}(x)\big)+d\big(f^{n_k+r}(x),f^{m_k+r}(x)\big)+d\big(f^{m_k+r}(x),f^{m_k}(x)\big)\le\\ &\le d\w_{f^{n_k}}\big(d(f^r(x),x)\big)+d\w_{f^r}\big(d(f^{n_k}(x),f^{m_k}(x))\big)+
d\w_{f^{m_k}}\big(d(f^r(x),x)\big)\le\\
&\le d\w_{f^{n_k}}(D)+d\w_{f^r}(4\e)+d\w_{f^{m_k}}(D)\le 3\e
\end{aligned}
$$
which is a contradiction.
\end{proof}

This lemma implies the following generalization of the Banach Contraction Principle.

\begin{theorem}\label{t3.4} Each eventually contracting continuous map $f:X\to X$ on a sequentially complete multimetric space $(X,\mathcal D)$ has a (unique) attracting fixed point.
\end{theorem}

\begin{proof} First we show that $f$ has at most one fixed point. Assuming that $x,y\in X$ are two distinct fixed points of $f$, we can find a pseudometric $d\in\mathcal D$ such that $d(x,y)>0$ and conclude that $d(x,y)=d(f^n(x),f^n(y))\le d\w_{f^n}(d(x,y))\not\to 0$, which contradicts the eventual contractivity of $f$.

By Lemma~\ref{l3.3}, for every $x\in X$ the sequence $\big(f^n(x)\big)_{n\in\w}$ is Cauchy with respect to the multimetric $\mathcal D$ and by the sequential completeness of $(X,\mathcal D)$ it converges to some point $x_\infty\in X$. Since
$$f(x_\infty)=f\big(\lim_{n\to\infty}f^n(x)\big)=\lim_{n\to\infty}f^{n+1}(x)=x_\infty,$$
the point $x_\infty$ is an attracting fixed point of $f$.
\end{proof}

\begin{remark} For eventually contracting maps on complete metric spaces, Theorem~\ref{t3.4} can be derived from Leader's Fixed Point Theorem \cite{Led83}, which states that a continuous map $f$ on a complete metric space $(X,d)$ has a contracting fixed point if for every $\e>0$ there are numbers $\delta>0$ and $r\in\IN$ such that for any numbers $i,j\in\w$ and points $x,y\in X$ with $d(f^i(x),f^j(y))<\e+\delta$ we get $d(f^{i+r}(x),f^{j+r}(y))<\e$.
\end{remark}

\section{Contracting function systems on multimetric spaces}\label{s4}

In this section we introduce various contractivity conditions for function systems on multimetric spaces. They are counterparts of  classical ``metric'' contraction properties considered in the previous section.

Let $(X,\mathcal D)$ be a multipseudometric space. For a family $\F\subset X^{\le X}$ of partial self-maps of $X$ and a pseudometric $d\in\mathcal D$ let $\dom(\F)=\bigcap_{f\in\F}\dom(f)$ be the {\em domain} of $\F$. By the {\em $d$-oscillation} of the family $\F$ we shall understand the function $d\w_\F:[0,\infty]\to[0,\infty]$ defined by
$$d\w_\F(t)=\sup(\{0\}\cup\{d(f(x),f(y)):f\in\F,\;x,y\in\dom(\F),\;d(x,y)\le t\}\mbox{ for }t\in[0,\infty].$$
It is clear that $d\w_\F\le \sup_{f\in\F}d\w_f(t)$; moreover, $d\w_\F=\sup_{f\in\F}d\w_f$ if $\dom(f)=\dom(\F)$ for all $f\in\F$.

\begin{definition} Let $(X,\mathcal D)$ be a multipseudometric space. A finite family of partial-self maps $\F\subset X^{\le X}$ is called
\begin{itemize}
\item {\em Banach  contracting} if $\sup_{0<t<\infty}{d\w_\F(t)}/{t}<1$ for every pseudometric $d\in\mathcal D$;
\item {\em Rakotch contracting} if $\sup_{a\le t<\infty}{d\w_\F(t)}/{t}<1$ for every pseudometric $d\in\mathcal D$ and every number $a>0$;
\item {\em Krasnoselski\u\i{}  contracting} if $\sup_{a\le t\le b}{d\w_\F(t)}/{t}<1$ for every pseudometric $d\in\mathcal D$ and every numbers $0<a<b<\infty$;
\item {\em Matkowski contracting} if  $\lim_{n\to\infty}d\w_\F^n(t)=0$ for every $d\in\mathcal D$ and $t>0$;
\item {\em eventually contracting} if  $\lim_{n\to\infty}d\w_{\F^n}(t)=0$ for every $d\in\mathcal D$ and $t>0$;
\item {\em Edelstein contracting} if for every $d\in\mathcal D$, $f\in\F$, and points $x,y\in \dom(\F)$ there is a constant $\lambda<1$ such that $d(f(x),f(y))\le\lambda\cdot d(x,y)$.
\end{itemize}
\end{definition}

It is easy to see that a finite family {$\F$} of partial self-maps of a multipseudometric space $(X,\mathcal D)$ is Banach (resp. Rakotch, Krasnoselski\u\i{}, Matkowski, Edelstein) contracting if so is each map $f\in\F$. In contrast, the eventual contractivity of a function system $\F\subset X^{\le X}$ is a stronger condition than the eventual contractivity of individual maps $f\in\F$.

\begin{example} For two continuous functions $f(x)=\max\{0,x-1\}$ and $g(x)=\min\{2,x+1\}$ on the interval $X=[0,2]$ the function system $\F=\{f,g\}$ is not eventually contracting while the functions $f$ and $g$ are eventually contractive{. More precisely, $f^2$ and $g^2$ are constant and  $(f\circ g)^n(x)=\min\{1,x\}$ for every $n\in\N$.}
\end{example}

For any function system $\F\subset X^{X}$ on a topological space $X$ the following implications hold:
\smallskip

\centerline{Banach  $\Ra$ Rakotch $\Ra$ Krasnoselski\u\i{} $\Ra$ Matkowski $\Ra$ Edelstein \&\ Eventual.}
\medskip

Moreover, if the space $X$ is compact, then the Edelstein contractivity is equivalent to the Rakotch contractivity of $\F$. This follows from our next proposition, which can be proved by analogy with Proposition~\ref{p3.2}.

\begin{proposition}\label{p4.2} Let $(X,\mathcal D)$ be a multipseudometric space and $\F\subset X^{\le X}$ be a finite family of partial self-maps of $X$.
\begin{enumerate}
\item If $\F$ is Banach contracting, then $\F$ is Rakotch contracting;
\item If $\F$ is Rakotch contracting, then $\F$ is Krasnoselski\u\i{} contracting;
\item If $\F$ is Krasnoselski\u\i{} contracting, then $\F$ is Matkowski contracting;
\item If $\F$ is Matkowski contracting, then $\F$ is Edelstein and eventually contracting;
\item The function system $\F$ is Edelstein contracting if and only if for every compact subset $K\subset \dom(\F)$ the system $\F|K=\{f|K:f\in\F\}$ is Rakotch  contracting.
\end{enumerate}
\end{proposition}

The $d$-oscillation $d\w_\F$ of a function system $\F\subset X^X$ on a pseudometric space $(X,d)$ upper bounds the $d_H$-oscillation $d_H\w_\F$ of the induced map $\F:\K(X)\to\K(X)$ on the hyperspace $\K(X)$ of $X$.
{
\begin{proposition}\label{p4.3} For any function system $\F\subset X^X$ on a topological space $X$ and a continuous pseudometric $d$ on $X$ we get $d_H\w_\F\le d\w_\F.$
\end{proposition}

\begin{proof} Given any $t>0$ we need to check that $d_H\w_\F(t)\le d\w_\F(t)$. Fix any two compact sets $A,B\in\K(X)$ at the Hausdorff distance $d_H(A,B)\le t$. For every $x\in\F(A)$ we can find a map $f\in\F$ and a point $a\in A$ such that $x=f(a)$. For the point $a\in A$ we can find a point $b\in B$ with $d(a,b)\le t$. Then $d(x,\F(B))\le d(f(a),f(b))\le d\w_f(d(a,b))\le d\w_\F(t)$.
By analogy we can prove that for every $y\in\F(B)$ we get $d(y,\F(A))\le d\w_\F(t)$, which implies
that
$$d_H(\F(A),\F(B))=\max\{\max_{x\in\F(A)}d(x,\F(B)),\max_{y\in\F(B)}d(y,\F(A))\}\le d\w_\F(t)$$and hence $d_H\w_\F(t)\le d\w_\F(t)$.
\end{proof}
}
\begin{corollary}\label{c4.4} Let $\F\subset X^X$ be a function system on a {(directed)} multimetric space $(X,\mathcal D)$. If $\F$ is Banach (resp. Rakotch, Krasnoselski\u\i{}, Matkowski, Edelstein, eventual{ly}) contracting, then so is the induced map $\F:\K(X)\to\K(X)$ on the multipseudometric {(multimetric)} space $(\K(X),\mathcal D_H)$.
\end{corollary}

\begin{theorem}\label{t4.5} For an eventual{ly} contracting function system $\F$ on a multimetric space $(X,\mathcal D)$ the following conditions are equivalent:
\begin{enumerate}
\item $\F$ is topologically contracting;
\item $\F$ is compactly contracting;
\item $\F$ is locally contracting;
\item $\F$ has an attractor;
\item $\F$ is compact-dominating;
\item there is a non-empty compact subset $K\subset X$ such that $\F(K)\subset K$.
\end{enumerate}
The equivalent conditions $(1)$--$(6)$ hold if the multimetric space $X$ is sequentially complete.
\end{theorem}

\begin{proof} The equivalence $(1)\Leftrightarrow(2)$ has been proved in Theorem~\ref{t7.3}.
The implication $(1)\Ra(4)$ follows from Mihail's Theorem~\ref{Mih} while {$(4)\Ra (5)\Ra(6)$ and $(3)\Ra(2)$ are trivial.} 
It remains to check that $(6)\Ra(3)$.

Assume that there is a compact subset $K_0\subset X$ such that $\F(K_0)\subset K_0$.
First we prove that $\F$ is compact-dominating. This will follow as soon as we check that for any compact set $C\subset X$ the set $K=K_0\cup\F^{<\w}(C)$ is compact. Given any open cover $\U$ of $X$, find a finite subfamily $\V\subset \U$ covering the compact set $K_0$. Next, use Lemma~\ref{l2.2} to find a finite subfamily $\mathcal E\subset\mathcal D$ and $\e>0$ such that for every $x\in K_0$ the ball $B_{\mathcal E}(x,\e)=\bigcap_{d\in\mathcal E}B_d(x,\e)$ is contained in some set $V\in\V$. Consider the pseudometric $\bar d=\max\mathcal E$. It follows that $\bar d\w_{\F^n}\le\max_{d\in\mathcal E}d\w_{\F^n}$ for every $n\in\IN$ and hence $\bar d\w_{\F^n}(t)\to 0$ for every $t\ge 0$.
In particular, for the number $t=\diam_{\bar d}(C\cup K_0)$ the sequence {$\big(\bar d\w_{\F^n}(t)\big)_{n\in\w}$} tends to zero and hence $\bar d\w_{\F^{\ge n}}(t)<\e$ for some $n\in\IN$. It follows that for every $f\in \F^{\ge n}$
$$\diam_{\bar d}f(K_0\cup C)\le \bar d\w_{\F^{\ge n}}(t)<\e,$$ which implies that the set $f(K_0\cup C)$ is contained in the ball $B_{\bar d}(x,\e)$ for any $x\in f(K_0)\subset K_0$ and hence $f(K_0\cup C)$ is contained in some set $V\in\V$. Consequently, $K_0\cup\F^{\ge n}(C)\subset\bigcup\V$. The set $\F^{<n}(C)=\bigcup_{k<n}\F^k(C)$, being compact, is covered by a finite subfamily $\W\subset\U$. Then $\V\cup\W\subset\U$ is a finite subcover of the set $K=K_0\cup\F^{\ge n}(C)\cup\F^{<n}(C)$, which implies that $K$ is compact. It is clear that $C\subset K$ and $\F(K)\subset K$. So, $\F$ is compact-dominating.

To show that $\F$ is locally contracting, fix a compact subset $K\subset X$ and an open cover $\U$ of $X$. We need to find a neighborhood $O_K\subset X$ and $n\in\IN$ such that $\{f(O_K):f\in\F^{\ge n}\}\prec\U$. Since $\F$ is compact-dominating, we can replace  $K$ by a larger compact set {and} assume that $\F(K)\subset K$. By Lemma~\ref{l2.2}, there are finite subfamily $\mathcal E\subset\F$ and $\e>0$ such that for each $x\in K$ the ball $B_{\mathcal E}(x,\e)$ is contained in some set $U\in\U$. Consider the pseudometric $\bar d=\max\mathcal E$ and the neighborhood $O_K=B_{\bar d}(K,1)$, which has finite $\bar d$-diameter $D=\diam_{\bar d}(O_K)\le\diam_{\bar d}(K)+2$. Since $\F$ is eventually contracting, there is a number $n\in\IN$ such that $\sup_{f\in\F^{\ge n}} d\w_f(D)<\e$ for all $d\in\mathcal E$. This implies that $\sup_{f\in\F^{\ge n}}\bar d\w_f(D)<\e$ and hence for every $f\in\F^{\ge n}$ the set $f(O_K)$ has $\bar d$-diameter ${\leq}\bar d\w_f(D)<\e$. Then $f(O_K)\subset B_{\bar d}(x,\e)=
B_{\mathcal E}(x,\e)\subset U$ for some $x\in f(K)\subset K$ and $U\in\U$. This completes the proof of the local contractibility of $\F$.
\smallskip

Now assuming that the multimetric space $(X,\mathcal D)$ is sequentially complete, we shall prove that the equivalent conditions (1)--(6) hold. It suffices to prove the condition (6).

Consider the directed multimetric $\bar{\mathcal D}=\{\max\mathcal E:\mathcal E\subset\mathcal D,\;\;|\mathcal E|<\infty\}$ on $X$ and observe that the function system $\F$ remains eventually contracting with respect to the multimetric $\bar{\mathcal D}$. Indeed, for any finite subset $\mathcal E\subset\mathcal D$ and the pseudometric $\bar d=\max\mathcal E$, we get $\bar d\w_{\F^n}\le\max\{d\w_{\F^n}:d\in\mathcal E\}$ and hence  $\bar d\w_{\F^n}\to 0$.

By Proposition~\ref{p2.1}, the Vietoris topology on the hyperspace $\K(X)$ is generated by the multimetric $\bar{\mathcal D}_H=\{d_H:d\in\bar{\mathcal D}\}$. By Corollary~\ref{c4.4}, the map $\F:\K(X)\to\K(X)$ is eventually contracting with respect to the multimetric $\bar{\mathcal D}_H$ and by Lemma~\ref{l3.3}, for every compact set $K\in\K(X)$ the sequence $(\F^n(K))_{n\in\w}$ is Cauchy in the multimetric space $(\K(X),\bar{\mathcal D}_H)$. Consequently, the set $\F^{<\w}(K)=\bigcup_{n\in\w}\F^n(K)$ is bounded with respect to each pseudometric $d\in\mathcal D$.

Now take any point $x\in X$ and for every sequence $\vec f=(f_n)_{n\in\w}\in\F^\w$ consider the sequence $(f_0\circ \dots\circ f_n(x))_{n\in\w}$. We claim that this sequence is Cauchy with respect to the multimetric $\mathcal D$. Indeed, for any pseudometric $d\in\mathcal D$ and any $\e>0$ we can find $k\in\IN$ so large that {$d\w_{\F^r}\big(\diam_d(\F^{<\w}(x))\big)<\e$ for every $r\geq k$. Then for any numbers $m\ge n\ge k$ we get}
$$d(f_0\circ\dots\circ f_n(x),f_0\circ\dots\circ f_m(x))\le d\w_{\F^{n+1}}(d(x,f_{n+1}\circ\dots\circ f_m(x))\le d\w_{\F^{n+1}}(\diam_d\F^{<\w}(x))<\e,$$
which means that the sequence $\big(f_0\circ \dots\circ f_n(x)\big)_{n\in\w}$ is Cauchy and by the sequential completeness converges to some point $\pi_x(\vec f)\in X$. The map $\pi_x:\F^\w\to X$, $\pi_x:\vec f\mapsto \pi_x(\vec f)$, is continuous since for any sequences $\vec f=(f_n)_{n\in\w}$ and $\vec g=(g_n)_{n\in\w}$ such that $(f_0,\dots,f_n)=(g_0,\dots,g_n)$ for some $n\in\IN$ we get
$$d(\pi_x(\vec f),\pi_x(\vec g))\le d\w_{\F^{n+1}}(\diam_d\F^{<\w}(x))\to 0$$ as $n\to\infty$.
The continuity of the map $\pi_x$ implies that the subset $K=\pi_x(\F^\w)$ is compact. It is easy to check that $\F(K)\subset K$, which means that the condition (6) holds.
\end{proof}

\begin{remark} Applying Theorem~\ref{t4.5} to a function system $\F=\{f\}$ consisting of a single eventually contracting continuous map $f:X\to X$ on a sequentially complete multimetric space $(X,\mathcal D)$, we conclude that $f$ has an attracting fixed point. So, Theorem~\ref{t3.4} is a partial case of Theorem~\ref{t4.5}. On the other hand, Theorem~\ref{t4.5} cannot be derived from Theorem~\ref{t3.4} applied to the hyperspace $(\K(X),\mathcal D_H)$ as the latter multimetric space is not necessarily sequentially complete (see Example~\ref{ex2.2}).
\end{remark}

For Edelstein contracting function systems we have the following criterion.

\begin{theorem}\label{t4.7} For an Edelstein contracting function system $\F$ on a multimetric space $(X,\mathcal D)$ the following conditions are equivalent:
\begin{enumerate}
\item $\F$ is topologically contracting;
\item $\F$ is compactly contracting;
\item $\F$ is locally contracting;
\item $\F$ has an attractor;
\item $\F$ is compact-dominating.
\end{enumerate}
\end{theorem}

\begin{proof} The implications {$(3)\Ra(2)$ and $(4)\Ra(5)$ are trivial, and $(2)\Leftrightarrow(1)$ was proved in Theorem~\ref{t7.3}. The implication $(1)\Ra(4)$ follows from Mihail's Theorem~\ref{Mih}.}

$(1)\Ra(3)$ Assume that $\F$ is topologically contracting. By Theorem~\ref{t7.3}, $\F$ is compactly contracting. To show that $\F$ is locally contracting, fix a compact subset $K\subset X$ and an open cover $\U$ of $X$. Since $\F$ is compact dominating, we can replace $K$ by a larger compact set and assume that $\F(K)\subset K$. By Lemma~\ref{l2.2}, there are finite subfamily $\mathcal E\subset\F$ and $\e>0$ such that for each $x\in K$ the ball $B_{\mathcal E}(x,2\e)$ is contained in some set $U\in\U$. Being topologically contracting, the system $\F$ has an attractor $A_\F$. Consider the open cover $\V=\{X\setminus A_\F\}\cup\{B_{\mathcal E}(x,\e):x\in\A_\F\}$ of $X$. Since $\F$ is compactly contracting, for some $n\in\IN$ the family $\{f(K):f\in\F^{\ge n}\}$ refines the cover $\V$. Let $O_K=B_{\bar d}(K,\e)$ where $\bar d=\max\mathcal E$. We claim that $\{f(O_K):f\in\F^{\ge n}\}\prec\U$. Take any function $f\in\F^{\ge n}$ and consider the set $f(K)$, which contains the set $f(A_\F)\subset A_\F$ and hence cannot be contained in the set $X\setminus A_\F$. By the choice of $n$, the set $f(K)$ is contained in some ball $B_{\bar d}(x,\e)\in\mathcal V$, $x\in A_\F$. Since $\F$ consists of Edelsten contractive maps, $f(O_K)=f(B_{\bar d}(K,\e))\subset B_{\bar d}(f(K),\e)\subset B_{\bar d}(B_{\bar d}(x,\e),\e)\subset B_{\bar d}(x,2\e)\subset U$ for some $U\in\U$.
\smallskip

$(5)\Ra(1)$ Assume that $\F$ is compact-dominating and take any non-empty compact set $K\subset X$ with $\F(K)\subset K$. By Proposition~\ref{p4.2}{(5)}, the function system $\F|K=\{f|K:f\in\F\}$ is Rakotch  (and hence eventually) contracting, and by Theorem~\ref{t4.5}, it is topologically contracting, which implies that for every sequence $(f_n)_{n\in\w}\in\F^\w$ the intersection $\bigcap_{n\in\w}f_0\circ\dots\circ f_n(K)$ is a singleton. This means that $\F$ is topologically contracting.
\end{proof}

\section{Metrization of contracting function systems on Tychonoff spaces}\label{s9}

In this section we address the following problem:

\begin{problem}\label{pr6.1} Detect function systems $\F$ on Tychonoff spaces $X$ which are Edelstein (Matkowski, Krasnoselski\u\i{}, Banach) contracting with respect to a suitable admissible multimetric on $X$.
\end{problem}

\subsection{Constructing contractive pseudometrics}
In this section we shall describe a general construction transforming a continuous pseudometric $d$ on a topological space $X$ endowed with a topologically contracting function system $\F\subset X^X$ into a $k$-continuous pseudometric $\hat d\ge d$ making the function system $\F$ Edelstein contracting.

A pseudometric $d$ of a topological space $X$ will be called {\em $k$-continuous} if for each compact subset $K\subset X$ the restriction $d|K\times K$ is continuous. It is clear that each continuous pseudometric is $k$-continuous. The converse is true if $X$ is a $k$-space. This follows from the fact that the identity map $X\to (X,d)$, being continuous on compacta, is continuous.

The following construction develops the ideas of  Barnsley and Igudesman \cite{BI} who modified metrics to make a given function system non-expanding (i.e., with Lipschitz constant $\le 1$). The same formula for $\hat d$ was independently discovered by {Mihail and} Miculescu \cite{Micu}.

\begin{proposition}\label{p9.1} Let $\F\subset X^X$ be a compactly contracting function system on a topological space $X$ and $(\alpha_n)_{n=0}^\infty$ be a strictly increasing sequence of positive real numbers with $\alpha_0=1$ and $\alpha_\w=\sup_{n\in\w}\alpha_n \le2$. Given a continuous pseudometric $d$ on $X$, consider the pseudometric $\hat d$ defined by
$$\hat d(x,y)=\sup_{n\in\w}\max_{f\in\F^{ n}}\alpha_n d(f(x),f(y)),\;\;x,y\in X.$$
The pseudometric $\hat d$ has the following properties:
\begin{enumerate}
\item $d\leq\hat{d}\leq 2\cdot \diam_d(X)$;
\item $\F$ is Edelstein contracting with respect to the pseudometric $\hat d$;
\item the pseudometric $\hat d$ is $k$-continuous;
\item the pseudometric $\hat d$ is continuous if $\F$ is locally contracting;
\item If $\F$ is {eventually} contracting with respect to the pseudometric $d$, then $\F$ is Krasnoselski\u\i{} contracting with respect to the pseudometric $\hat d$;
\item If $\F$ is globally contracting, then $\F$ is Rakotch contracting with respect to the pseudometric $\hat d$;
\item If $\F$ is globally contracting and the pseudometric $d$ is totally bounded, then $\hat d$ is totally bounded too.
\end{enumerate}
\end{proposition}

\begin{proof}
1. The inequality $d\le\hat d\le 2\cdot \diam_d(X)$ follows immediately from the definition of the pseudometric $\hat d$ and the equality $\alpha_0=1$.
\smallskip

2. Next, we show that $\F$ is Edelstein contracting with respect to the pseudometric $\hat d$. Given any function $f\in\F$ and points $x,y\in X$ we should find a real number $\lambda<1$ such that $\hat d(f(x),f(y))\le\lambda\cdot \hat d(x,y)$. If $\hat d(f(x),f(y))=0$, then we can take any $\lambda<1$.

So we assume that $\hat d(f(x),f(y))>0$. The compact contractivity of $\F$ implies that for any compact $K\subset X$ we get $\lim_{n\to\infty}\sup_{f\in\F^{\geq n}}\diam_d(f(K))=0$. This implies that $\hat d(f(x),f(y))=\alpha_nd(g\circ f(x),g\circ f(y))$ for some $n\in\w$ and some $g\in\F^{n}$ (which means that in the definition of $\hat{d}$ we can replace ``$\sup$" by ``$\max$"). Since $\alpha_n<\alpha_{n+1}$, we conclude that
$$0<\hat d(f(x),f(y))=\alpha_nd(g\circ f(x),g\circ f(y))<\alpha_{n+1}d(g\circ f(x),g\circ f(y))\le \hat d(x,y).$$
\smallskip

3. To show that the pseudometric $\hat d$ is $k$-continuous, it suffices for any compact set $K\subset X$, point $x\in K$ and $\e>0$ to find a neighborhood {$O_x\subset X$ such that $O_x\cap K\subset B_{\hat d}(x,\e)$}. Since $\F$ is compactly contracting there is a number $n\in\w$ such that
$\sup_{f\in\F^{\ge n}}\diam_{\hat d}(f(K))<\e/2$. By the continuity of the maps $f\in \F^{<n}=\bigcup_{m<n}\F^m$ there is a neighborhood $O_x\subset X$ of $x$ such that $\max_{f\in \F^{<n}}d(f(x),f(y))<\e/2$ for all $y\in O_x$. The choice of the number $n$ guarantees that for every {$y\in O_x\cap K$}, we get
$$\hat d(x,y)=\sup_{m\in\w}\underset{f\in \F^m}{{\max}}\alpha_md(f(x),f(y))<\e.$$

4. Assume that the function system $\F$ is locally contracting. Given any point $x\in X$ and $\e>0$ we need to find a neighborhood $O_x\subset X$ of $x$ such that $O_x\subset B_{\hat d}(x,\e)$.
Since $\F$ is locally contracting, there are a number $n\in\w$ and a neighborhood $U_x\subset X$ such that
$\sup_{f\in\F^{\ge n}}\diam_{\hat d}(f(U_x))<\e/2$. By the continuity of the maps $f\in\bigcup_{m<n}\F^m$ there is a neighborhood $O_x\subset U_x$ of $x$ such that $$\max_{m<n}\max_{f\in \F^m}\alpha_m d(f(x),f(y))<\e$$ for all $y\in O_x$. The choice of the number $n$ guarantees that for every $y\in O_x$ we get
$$\hat d(x,y)=\sup_{m\in\w}\underset{f\in \F^m}{{\max}}\alpha_md(f(x),f(y))<\e.$$
\smallskip

5. Assume that $\F$ is eventually contracting with respect to the pseudometric $d$. To show that $\F$ is Krasnoselski\u\i{} contracting with respect to the pseudometric $\hat d$, we need to check that $\sup_{r\le t\le R}\hat d\w_\F(t)/t<1$ for any positive real numbers $r<R$. The eventual contractivity of $\F$ yields a number $k\in\w$ such that $d\w_{\F^{\ge k}}(R)<\frac14r$. Consider the number
$$\lambda=\max_{m\le k}\frac{\alpha_m}{\alpha_{m+1}}\in[\tfrac12,1).$$
We claim that $\hat d\w_\F(t)/t\le \lambda$ for every $t\in[r,R]$. Assuming that $\hat d\w_\F(t)/t>\lambda$, we can find a function $f\in\F$ and two points $x,y\in X$ such that $\hat d(x,y)\le t$ and $\hat d(f(x),f(y))>\lambda t$. Since $d(x,y)\le\hat d(x,y)\le t\le R$, the choice of $k$ guarantees that $\sup_{g\in\F^{\ge k}}d(g\circ f(x),g\circ f(y))\le\sup_{g\in\F^{> k}}d(g(x),g(y))\le\sup_{g\in\F^{>k}}d\w_g(R)<\frac14r$ and hence
$${
\begin{aligned}
&\sup_{m\ge k}\alpha_m\cdot\max_{g\in\F^m}d(g\circ f(x),g\circ f(y))\le \alpha_\w\cdot\sup_{g\in\F^{\ge k}}d(g\circ f(x),g\circ f(y))\le \frac{2}{4}r\leq\lambda r\le\lambda t
\end{aligned}}
$$and { hence
$$
\begin{aligned}
&\hat{d}(f(x),f(y))=\max_{m<k}\max_{g\in\F^m}\alpha_md(g\circ f(x),g\circ f(y))=\max_{m<k}\frac{\alpha_m}{\alpha_{m+1}}\cdot\max_{g\in\F^{m}}\alpha_{m+1}d(g\circ f(x),g\circ f(y))\le \\
&\lambda
\cdot \max_{m<k}\max_{g\in\F^{m}}\alpha_{m+1}d(g\circ f(x),g\circ f(y))\le \lambda\cdot  \hat d(f(x),f(y))<\hat d(f(x),f(y)).
\end{aligned}
$$}
This contradiction shows that {$\sup_{t\in[r,R]}d\w_\F(t)/t\le\lambda<1$}, which means that $\F$ is Krasnoselski\u\i{} contracting.
\smallskip

6. Assume that $\F$ is globally contracting. To show that $\F$ is Rakotch contracting with respect to the pseudometric $\hat d$, it suffices for every $\delta>0$ to find $\lambda<1$ such that
for every $t\ge\delta$ and points $x,y\in X$ with $\hat d(x,y)\le t$ we get $\hat d(f(x),f(y))\le\lambda\cdot t$ for all $f\in\F$.

Since $\F$ is globally contracting, there is a number $k\in\w$ such that for every $f\in\F^{\ge k}$ the set $f(X)$ has $\diam_d f(X)<\frac{\delta}{4}$. Choose any real number $\lambda\in[\frac12,1)$ such that $\lambda\ge \frac{\alpha_{n}}{\alpha_{n+1}}$ for all $n\le k$. To show that $\lambda$ satisfies our requirements, take any $t\ge\delta$ and points $x,y\in X$ with $\hat d(x,y)\le t$. We need to show that $\hat d(f(x),f(y))\le \lambda t$ for every $f\in\F$. It follows that $\hat d(f(x),f(y))=\alpha_n\cdot d(g\circ f(x),g\circ f(y))$ for some $n\in\w$ and $g\in\F^n$. If $n\le k$, then
$$\hat d(f(x),f(y))=\alpha_n d(g\circ f(x),g\circ f(y))=\frac{\alpha_n}{\alpha_{n+1}}\alpha_{n+1}d(g\circ f(x),g\circ f(y))\le \lambda \hat d(x,y)\le\lambda t.$$
If $n>k$, then $$\hat d(f(x),f(y))=\alpha_n\cdot d(g\circ f(x),g\circ f(y))\le \alpha_\w\cdot \frac{\delta}{4}\le\frac12\delta\le\lambda t$$by the choice of $k$.
\smallskip

7. Finally assume that $\F$ is globally contracting and the pseudometric $d$ is totally bounded. Given any $\e>0$, we need to find a finite subset $F\subset X$ such that $X=\bigcup_{x\in F}B_{\hat{d}}(x,\e)$. By the global contractivity of $\F$, there is a number $k\in\IN$ such that for each function $f\in\F^{\ge k}$ the set $f(X)$ has diameter $\diam_d(f(X))<\frac{\e}{8}$. The total boundedness of the pseudometric $d$ implies the total boundedness of the pseudometric $\hat d_k(x,y)=\max_{n< k}\max_{f\in\F^n}\alpha_nd(f(x),f(y))$.
Consequently, there is a finite set $F\subset X$ such that $X=\bigcup_{x\in F}B_{\hat d_k}(x,\e/2)$. Since
$|\hat d-\hat d_k|\le \frac\e4$, we conclude that $X=\bigcup_{x\in F}B_{\hat d}(x,\e)$, which means that the metric $\hat d$ is totally bounded.
\end{proof}

\subsection{Metrization of locally and topologically  contracting function systems}

\begin{theorem}\label{t9.2} For a compact-dominating function system $\F$ on a Tychonoff space $X$ the following conditions are equivalent:
\begin{enumerate}
\item $\F$ is locally contracting;
\item $\F$ is Edelstein  contracting with respect to some admissible multimetric $\mathcal D$ on $X$;
\item $\F$ is Edelstein  contracting with respect to some  admissible bounded multimetric $\mathcal D$ on $X$ of cardinality $|\mathcal D|=\mn(X)$.
\end{enumerate}
If $X$ is a $k$-space, then the conditions \textup{(1)---(3)} are equivalent to
\begin{itemize}
\item[(4)] $\F$ is topologically contracting;
\item[(5)] $\F$ is compactly contracting.
\end{itemize}
If the Tychonoff space $X$ is sequentially complete, then the conditions \textup{(1)--(3)} are equivalent to
\begin{itemize}
\item[(6)] $\F$ is Edelstein  contracting with respect to some sequentially complete admissible bounded multimetric $\mathcal D$ on $X$.
\end{itemize}
\end{theorem}

\begin{proof} The implication $(3)\Ra(2)$ is trivial.

$(2)\Ra(1)$. Given a compact set $K\in\K(X)$ and an open cover $\U$ of $X$, we need to find a neighborhood $O_K\subset X$ and a number ${n}\in\IN$ such that $\{f(O_K):f\in\F^{\ge n}\}\prec\U$. By Theorem~\ref{t4.7}, $\F$ has an attractor $A_\F$, and we can assume that $A_\F\subset K$ and $\F(K)\subset K$. By Lemma~\ref{l2.2}, there are a finite subfamily $\mathcal E\subset\mathcal D$ and $\e>0$ such that each ball $B_{\mathcal E}(x,2\e)$, $x\in A_\F$, is contained in some set $U\in\U$. By Proposition~\ref{p4.2}, the restriction $\F|K$ is Matkowski contracting, which implies that $\hat d\w_{\F|K}^n\to 0$ for the pseudometric $\hat d=\max {\mathcal E}$.
Then we can find $n\in\w$ such that $\hat d\w_{\F|K}^n(\diam_{\hat d}(K))<\e$. It follows that for every $m\ge n$ and $f\in \F^m$ the set $f(K)$ has $\hat d$-diameter
$$\diam_{\hat d}(f(K))\le\hat d\w_{\F|K}^m(\diam_{\hat d}(K))<\e.$$
Since $f(A_\F)\subset \F^m(A_\F)=A_\F$, the set $f(K)$ is contained in the ball $B_{\hat d}(x,\e)$ for any point $x\in f(A_\F)\subset A_\F$. Now consider the neighborhood
$$O_K=B_{\hat d}(K,\e):=\{x\in X:\inf_{y\in K}\hat d(x,y)<\e\}$$ of $K$ and observe that for every $f\in\F^{\ge n}$ we get
$$f(O_K)= f(B_{\hat d}(K,\e))\subset B_{\hat d}(f(K),\e)\subset B_{\hat d}(B_{\hat d}(x,\e),\e)\subset B_{\hat d}(x,2\e)=B_{\mathcal E}(x,2\e)$$ and hence
$$\{f(O_K):f\in \F^{\ge n}\}\prec\{B_{\mathcal{E}}(x,2\e):x\in A_\F\}\prec\U.$$
\smallskip

$(1)\Ra(3)$. Fix any bounded admissible multimetric $\mathcal D$ of cardinality $|\mathcal D|=\mn(X)$ on $X$. By Proposition~\ref{p9.1}, for every pseudometric $d\in\mathcal D$ there is a bounded continuous pseudometric $\hat d\ge d$ on $X$ with respect to which the function family $\F$ is Edelstein contracting. Then $\hat{\mathcal D}=\{\hat d:d\in\mathcal D\}$ is an admissible bounded multimetric on $X$ of cardinality $\mn(X)$ (because $\mn(X)\le |\hat{\mathcal D}|\le|\mathcal D|=\mn(X)$) making the function system $\F$ Edelstein contracting.
\smallskip

By Theorem~\ref{t7.3} the conditions (4) and (5) are equivalent. If $X$ is a $k$-space, then the equivalence $(1)\Leftrightarrow(5)$ follows from Theorem~\ref{t5.3}.
\smallskip

Now assuming that $X$ is sequentially complete, we shall prove the equivalence of the conditions {$(6)$} and $(2)$. In fact, the implication ${(6)}\Ra(2)$ is trivial.
To prove that $(2)\Ra{(6)}$, consider the universal multimetric $\mathcal D_X$ on $X$. By the sequential completeness of $X$, the multimetric space $(X,\mathcal D_X)$ is sequentially complete.
By Proposition~\ref{p9.1}, for every pseudometric $d\in\mathcal D_X$ we can find a bounded continuous pseudometric
$\hat d\ge \min\{1,d\}$ on $X$ with respect to which $\F$ is Edelstein contracting. Then $\hat{\mathcal D}=\{\hat d:d\in\mathcal D_X\}$ is an admissible {bounded} multimetric on $X$ making the function system $\F$ Edelstein contracting. For every $d\in\mathcal D_X$ the inequality $\hat d\ge \min\{1,d\}$ implies that each Cauchy sequence $(x_n)$ in the multimetric space $(X,\hat{\mathcal D})$ remains Cauchy with respect to the multimetric $\mathcal D$. The sequential completeness of the multimetric $\mathcal D_X$ guarantees that the Cauchy sequence $(x_n)$ converges in $X$, witnessing that the multimetric space $(X,\hat{\mathcal D})$ is sequentially complete.
\end{proof}

A topological space is called {\em completely metrizable} if its topology is generated by a complete metric. Proposition~\ref{p9.1} implies:

\begin{corollary} Each topologically contracting function system $\F$ on a (completely) metrizable topological space $X$ is Edelstein contracting with respect to some (complete) admissible metric on $X$.
\end{corollary}

{
The following example shows, among other items, that in the above corollary, we can not replace ``Edelstein'' by ``eventually''. In particular, in the condition (3) of Theorem \ref{t9.2} we also cannot replace ``Edelstein'' by ``eventually''.}

\begin{example}\label{fex} On the countable product of lines $\IR^\w$ consider the  function $f:\IR^\w\to\IR^\w$, $f:(x_n)_{n\in\w}\mapsto (\frac12x_n)_{n\in\w}$. The function system $\F=\{f\}$ has the following properties:
\begin{enumerate}
\item $\F$ is locally contracting;
\item $\F$ is not totally contracting;
\item $\F$ is Banach  contracting for the admissible multimetric $\mathcal D=\{d_n\}_{n\in\w}$ on $\IR^\w$ consisting of the pseudometrics $d_n((x_k)_{k\in\w},(y_k)_{k\in\w})=|x_n-y_n|$, $n\in\w$;
\item $\F$ is Edelstein contracting for some admissible metric on $\IR^\w$;
\item $\F$ is eventually contracting for no continuous metric on $\IR^\w$.
\end{enumerate}
\end{example}

\begin{proof} The Banach contractivity of $f$ with respect to the pseudometrics $d_n$ trivially follows from the definitions. Theorem~\ref{t9.2} implies that $\{f\}$ is locally contracting and is Edelstein contracting for some admissible metric on $\IR^\w$.

To see that $\{f\}$ is not totally contracting, take any neighborhood $U\subset\IR^\w$ of the fixed point $\mathbf 0$ of $f$. Then $U$ contains the set $A=\{0\}^{n_0}\times\IR^{\w\setminus n_0}$ for some $n_0\in\w$. Since $f(A)=A$, we have that $A\subset f^n(U)$ for any $n\in\N$, so the open cover $\U=\{(-k,k)^{n_0+1}\times \IR^{\w\setminus n_0+1}:k\in\N\}$ witnesses to the fact that $\F$ is not totally contracting.

To prove the final statement, assume that the map $f$ is eventually contracting for some continuous metric $d$ on $X$. By the continuity of the metric $d$ and the definition of the Tychonoff product topology on $\IR^\w$, there is a number $n\in\w$ such that the set $A=\{0\}^n\times\IR^{\w\setminus n}$ is contained in the 1-ball $B_d(\mathbf 0,1)$ centered at the fixed point $\mathbf 0=(0,0,\dots)$ of the map $f$. Choose $\e>0$ so small that $A\not\subset B_d(\mathbf 0,\e)$. Since $\lim_{k\to\infty}d\w_{f^k}(1)\to 0$, there is $m\in\w$ so large that $d\w_{f^m}(1)<\e/2$. Then $A=f^m(A)\subset f^m(B_d(\mathbf 0,1))\subset B_d(\mathbf 0,2\cdot d\w_{f^m}(1))\subset B_d(\mathbf 0,\e)$, which contradicts the choice of $\e$.
\end{proof}

\begin{problem} Is the statement (2) of Theorem \ref{t9.2} equivalent to one of the statements:
\begin{itemize}
\item[$(2a)$] $\F$ is eventually contracting with respect to some admissible multimetric $\mathcal D$ on $X$;
\item[$(2b)$] $\F$ is Matkowski contracting with respect to some admissible multimetric $\mathcal D$ on $X$; or
\item[$(2c)$] $\F$ is Banach contracting with respect to some admissible multimetric $\mathcal D$ on $X$?
\end{itemize}
\end{problem}

\subsection{Metrization of globally contracting function systems}

In this section we consider the problem of detecting function systems which are Matkowski or Rakotch contracting with respect to some admissible (totally) bounded multimetric. Let us recall that a multimetric $\mathcal D$ on a set $X$ is {\em totally bounded} if each pseudometric $d\in\mathcal D$ is totally bounded.

\begin{theorem}\label{ft1} For a compact-dominating function system $\F$ on a Tychonoff space $X$ the following conditions are equivalent:
\begin{enumerate}
\item $\F$ is globally contracting;
\item $\F$ is Rakotch contractive with respect to some bounded admissible multimetric $\D$ on $X$ with $|\D|=\mn(X)$;
\item $\F$ is Rakotch contractive with respect to some totally bounded admissible multimetric $\D$ on $X$ with $|\D|=\tmn(X)$;
\item $\F$ is eventually contractive with respect to some bounded admissible multimetric $\mathcal D$ on $X$.
\end{enumerate}
If the Tychonoff space $X$ is sequentially complete, then the conditions \textup{(1)--(4)} are equivalent to:
\begin{itemize}
\item[(5)] $\F$ is Rakotch contractive with respect to some sequentially complete bounded admissible multimetric $\D$ on $X$.
\end{itemize}
\end{theorem}

\begin{proof}
$(1)\Ra(2)$ (and $(1)\Ra(3)$). Assume that the function system $\F$ is globally contracting. Fix a (totally) bounded admissible multimetric $\mathcal D$ on $X$ of cardinality $\mn(X)$ (resp. $\tmn(X)$).

{By Proposition \ref{p9.1}(6,7), there is a family $\hat{\D}$ of continuous (totally) bounded pseudometrics such that $\F$ is Rakotch contractive with respect to $\hat{D}$, and $|\hat{D}|=\mn(X)$ (resp. $|\hat{D}|=\tmn(X)$).}
\smallskip

The implications $(2)\Ra(4)\Leftarrow(3)$ are trivial.
\smallskip

$(4)\Ra(1)$ Assume that $\F$ is eventually contracting with respect to some bounded admissible multimetric $\mathcal D$ on $X$. Given an open cover $\U$ of $X$, we need to find a number $n\in\w$ such that $\{f(X):f\in\F^{\ge n}\}\prec\U$. By Theorem~\ref{t4.5}, $\F$ has an attractor $A_\F$ and by Lemma~\ref{l2.2}, we can find a finite subfamily $\mathcal E\subset\mathcal D$ and $\e>0$ such that for each $x\in A_\F$ the ball $B_{\mathcal E}(x,\e)$ is contained in some set $U\in\U$. Since $\mathcal D$ is a bounded multimetric on $X$, the pseudometric $\bar d=\max\mathcal E$ on $X$ is bounded.

Since $\F$ is eventually contracting, for every $d\in \mathcal E$ there is $n_d\in\w$ such that $d\w_f(\diam_{\bar d}(X))<\e$ for every $f\in\F^{\ge n_d}$. Then for $n=\max_{d\in\mathcal E}n_d$ and any function $f\in\F^{\ge n}$ we get $\diam_{\bar d}(f(X))\le \max_{d\in\mathcal E}d\w_f(\diam_d(X))<\e$. Consequently, for every $x\in {f(A_\F)\subset f(X)\cap A_\F}$, we get $f(X)\subset
B_{\bar d}(x,\e)=B_{\mathcal E}(x,\e)\subset U$ for some $U\in\U$. This means that the function system $\F$ is globally contracting.
\smallskip

Now assuming $X$ is sequentially complete, we shall prove the implications $(1)\Ra(5)\Ra(4)$. In fact, $(5)\Ra(4)$ is trivial. It remains to prove $(1)\Ra (5)$. By Proposition~\ref{p9.1}(6), for every continuous pseudometric
$d\in\mathcal D_X$ on $X$, there is a bounded continuous pseudometric $\hat d\ge \min\{1,d\}$ making the function system $\F$ Rakotch contracting. Then $\F$ is Rakotch contracting with respect the bounded admissible pseudometric
$\hat{\mathcal D}=\{\hat d:d\in\mathcal D_X\}$. Repeating the argument from the proof of Theorem~\ref{t9.2}(6), we can show that the sequential completeness of the universal multimetric $\mathcal D_X$ of $X$ implies the seqeuntial completeness of the multimetric $\hat{\mathcal D}$.
\end{proof}

Taking into account the equivalence of various contractivity properties on compact spaces and applying Theorem~\ref{ft1}, we get:

\begin{theorem}\label{comp} For a function system $\F$ on a compact topological space $X$ the following conditions are equivalent:
\begin{enumerate}
\item $\F$ is topologically contracting;
\item $\F$ is compact{ly} contracting;
\item $\F$ is locally contracting;
\item $\F$ is totally contracting;
\item $\F$ is globally contracting;
\item $\F$ is eventually contracting with respect to some admissible multimetric $\mathcal D$ on $X$;
\item $\F$ is Edelstein contracting with respect to some admissible multimetric $\mathcal D$ on $X$;
\item $\F$ is Rakotch contracting with respect to some admissible multimetric $\mathcal D$ on $X$ with $|\mathcal D|=\mn(X)$;
\end{enumerate}
\end{theorem}


\subsection{Remetrization of eventually contracting function systems}

In this section we are interested in characterizing function systems which are eventually or Krasnoselski\u\i{} contracting with respect to some admissible multimetric. For function systems on metric spaces the necessary condition is the total contractivity.

\begin{theorem}\label{t6.9} Each eventually contracting function system $\F$ on a {complete} metric space $(X,d)$ is totally contracting.
\end{theorem}

\begin{proof} Given any compact subset $K\subset X$ { we need to find a neighborhood $O_K\subset X$ of $K$ such that for any open cover $\U$ of $X$, there is} a number $n\in\IN$ such that $\{f(O_K):f\in\F^{\ge n}\}\prec\U$.
{By Theorem~\ref{t4.5}, the function system $\F$ has an attractor $A_\F$. Clearly, we can assume that $A_\F\subset K$.
Let $O_K=\{x\in X:\inf_{y\in K}d(x,y)<1\}$ be the 1-neighborhood of $K$ in $X$.} By the compactness of $A_\F$ we can find $\e>0$ such that each ball $B_d(x,\e)$, $x\in A_\F$, is contained in some set $U\in\U$. Since $\F$ is eventually contracting, there is $n\in\IN$ such that $d\w_{\F^{\ge n}}(\diam_d(K)+2)<\e$. Then for every $f\in\F^{\ge n}$ the set $f(O_K)$ has diameter {$\diam_d(f(O_K))\le d\w_{\F^{\geq n}}(\diam_d(O_K))\le d\w_{\F^{\geq n}}(\diam_d(K)+2)<\e$} and hence is contained in some ball $B_d(x,\e)$ centered at a point {$x\in f(A_\F)\subset A_\F\cap f(O_K)$.} Consequently,
$\{f(O_K):f\in\F^{\ge n}\}\prec\{B_d(x,\e):x\in A_\F\}\prec\U.$
\end{proof}

We do not know if Theorem~\ref{t6.9} can be reversed.

\begin{problem}\label{c10.1} Is each totally contracting function system $\F$ on a metrizable space $X$ eventually contracting with respect to some admissible (multi)metric on $X$?
\end{problem}

In the following theorem for a cover $\U$ of a space $X$ and a subset $A\subset X$ by $\St(A,\U)=\bigcup\{U\in\U:A\cap U\ne\emptyset\}$ we denote the $\U$-{\em star} of $A$ and $\St(\U)=\{\St(U,\U):U\in\U\}$ the {\em star} of the cover $\U$.

\begin{theorem}
For a compact-dominating function system $\F$ on a Tychonoff space $X$ the following conditions are equivalent:
\begin{enumerate}
\item $\F$ is eventually contracting with respect to some admissible multimetric $\mathcal D$ on $X$;
\item $\F$ is Matkowski contracting with respect to some admissible multimetric $\mathcal D$ on $X$;
\item $\F$ is Krasnoselski\u\i{} contracting with respect to some admissible multimetric $\mathcal D$ on $X$;
\item for every open set $W\subset X$ and a point $x\in W$ there is a sequence of $(\U_n)_{n\in\IZ}$ of open covers of $X$ such that:
\begin{itemize}
\item[(a)] $\St(x,\U_n)\subset W$ for some $n\in\IZ$;
\item[(b)] any points $x,y\in X$ are contained in some set $U\in\bigcup_{n\in\IZ}{\U_n}$;
\item[(c)] for every $n\in\Z$, $St(\U_{n})\prec \U_{n+1}$;
\item[(d)] for every $m>n$, there is $k\in\Z$ such that $\{f(U):f\in\F^{\geq k},\;U\in \U_m\}\prec \U_n$.
\end{itemize}
\item for every open set $W\subset X$ and a point $x\in W$ there is a sequence of open neighborhoods of the diagonal $(V_n)_{n\in\IZ}$ of $X\times X$ such that:
\begin{itemize}
\item[(a)] $\{y\in X:(x,y)\in V_n\}\subset W$ for some $n\in\IZ$;
\item[(b)] $X\times X=\bigcup_{n\in\IZ}V_n$;
\item[(c)] for every $n\in \Z$, $V_{n}\circ V_{n}\subset V_{n+1}$;
\item[(d)] for every $m>n$, there is $k\in\N$ such that $\{(f(x),f(y)):(x,y)\in V_m,\;f\in\F^{\geq k}\}\subset V_n$.
\end{itemize}
\end{enumerate}
\end{theorem}

\begin{proof} We shall prove the implications $(3)\Ra(2)\Ra(1)\Ra(4)\Ra(5)\Ra(1)\Ra(3)$ among which $(3)\Ra(2)\Ra(1)$ are trivial.
\smallskip

$(1)\Ra(4)$ Assume that $\F$ is eventually contracting with respect to some admissible multimetric $\mathcal D$ on $X$.
Replacing $\mathcal D$ by the family $\bar{\mathcal D}=\{\max\mathcal E:\mathcal E\subset\mathcal D,\;|\mathcal E|<\infty\}$, we can assume that $\mathcal D$ is directed.

Given an open set $W\subset X$ and a point $x_0\in W$, we need to construct a sequence $(\U_n)_{n\in\IZ}$ of open covers of $X$ satisfying the conditions (a)--(d) of the statement (4).  Since the directed family of pseudometrics $\mathcal D$ generates the topology of $X$, there is a pseudometric $d\in\mathcal D$ such that the 1-ball $B_d(x_0,1)=\{y\in X:d(x_0,y)<1\}$ centered at {$x_0$} is contained in the open set $W$. Now for every $n\in\IZ$ consider the open cover  $\U_n=\{B_d(x,3^{n}):x\in X\}$ of $X$ and observe that the sequence $(\mathcal U_n)_{n\in\IZ}$ satisfies the conditions (a)--(c). To check the condition (d), take any integer numbers $n<m$. Taking into account that $\lim_{k\to\infty}d\w_{\F^k}(3^{m+1})=0$, we can find $k\in\IZ$ such that $d\w_{\F^{\ge k}}(3^{{m}+1})<3^n$. Then for every  $f\in\F^{\ge k}$ and $U\in\U_m$ we get
$$\diam_d(f(U))\le d\w_{\F^{\ge k}}(\diam_d(U))\le {d}\w_{\F^{\ge k}}(3^{{m}+1})<3^n,$$ which implies that $f(U)\subset B_d(x,3^n)\in\U_n$ for any point $x\in f(U)$.
\smallskip

$(4)\Ra(5)$ Given an open set $W\subset X$ and a point $x_0\in W$, apply the condition ({4}) to find a sequence of open covers $(\U_n)_{n\in\IZ}$ satisfying the condition $(a)$--$(d)$ of the item (4). Then the sequence $(V_n)_{n\in\IZ}$ of open neighborhoods $V_n=\bigcup_{U\in\U_n}U\times U$ satisfies the conditions (a)--(d) of the item (5).
\smallskip

$(5)\Ra(1)$ For each open set $W\subset X$ and a point $x_0\in W$ fix a sequence $(V_n)_{n\in\IZ}$ of open neighborhoods of the diagonal of $X\times X$ satisfying the conditions (a)--(d) of the item (5). Replacing $(V_n)_{n\in\IZ}$ by $(V_{2n})_{n\in\IZ}$ if necessary, we can assume that (c) holds in the stronger form: $V_n\circ V_n\circ V_n\subset V_{n+1}$. In this case we can repeat the argument of Theorem~8.1.10 of \cite{Eng} and construct a continuous pseudometric $d=d_{W,x_0}$ on $X$ such that
$$\{(x,y):d(x,y)<2^n\}\subset V_n\subset \{(x,y):d(x,y)\le 2^n\}.$$
We claim that $\lim_{n\to\infty}\max_{f\in\F^{\ge {n}}}{d}\w_f(t)\to 0$ for every $t\in[0,\infty)$. Given any $t>0$ and $\e>0$, find two integer numbers  $n<m$ such that $2^n<\e$ and $2^m>t$. The condition (d) yields a number $k\in\IZ$ such that  $\{(f(x),f(y)):(x,y)\in V_m,\;f\in\F^{\geq k}\}\subset V_n$. We claim that $\sup_{f{\in}\F^{\ge k}}d\w_f(t)<\e$.
Take any function $f\in\F^{\ge k}$ and any set $A\subset X$ of $d$-diameter $\diam_d(A)\le t$. Then
$$A\times A\subset\{(x,y)\in X\times X:d(x,y)\le t\}\subset\{(x,y)\in X\times X:d(x,y)<2^m\}\subset V_m$$ and by the choice of $k$,
$$\{(f(x),f(y)):(x,y)\in A\times A\}\subset \{(f(x),f(y)):(x,y)\in V_m\}\subset V_n\subset\{(x,y)\in X\times X:d(x,y)\le 2^n\},$$
which implies that $\diam_d(f(A))\le 2^n<\e$. Hence $d\w_{\F^{\ge k}}(t)<\e$ and $\lim_{n\to\infty} d\w_{\F^n}(t)=0$.

Now we see that $\mathcal D=\{d_{W,x_0}:x_0\in W\in\tau_X\}$ is an admissible multimetric on $X$ and $\F$ is eventually contracting with respect to $\mathcal D$.
\smallskip

The implication $(1)\Ra(3)$ follows from Theorem~\ref{t6.12n} below.
\end{proof}

\begin{theorem}\label{t6.12n} For a cardinal $\kappa>0$ and a compact-dominating function system $\F$ on a Tychonoff space $X$ the following conditions are equivalent:
\begin{enumerate}
\item $\F$ is Krasnoselski\u\i{} contracting with respect to some (sequentially complete) admissible multimetric $\mathcal D$ on $X$ of cardinality $|\mathcal D|\le\kappa$;
    \item $\F$ is Matkowski contracting with respect to some (sequentially complete) admissible multimetric $\mathcal D$ on $X$ of cardinality $|\mathcal D|\le\kappa$;
\item $\F$ is eventually contracting with respect to some (sequentially complete) admissible multimetric $\mathcal D$ on $X$ of cardinality $|\mathcal D|\le\kappa$.
\end{enumerate}
\end{theorem}

\begin{proof} The implications $(1)\Ra(2)\Ra(3)$ are trivial.

$(3)\Ra(1)$ Assume that $\F$ is eventually contracting with respect to some admissible (sequentially complete) multimetric $\mathcal D$ on $X$ with $|\mathcal D|\le\kappa$.
Fix an increasing sequence of real numbers $(\alpha_n)_{n\in\w}$ such that $1\leq \alpha_n\leq 2$ for all $n\in\w$. Proposition~\ref{p9.1}(5) guarantees that for every pseudometric $d\in\mathcal D$ the function system $\F$ is Krasnoselski\u\i{} contracting with respect to the pseudometric
$$
\hat{d}(x,y)=\sup_{n\in\w}\alpha_n\cdot\max_{f\in\F^{n}}d(f(x),f(y)),
$$
which is well-defined and continuous. The continuity of $\hat d$ follows from Proposition~\ref{p9.1}(4) and Theorem~\ref{t4.5}. The continuity of the pseudometrics $\hat d\ge d$, $d\in\mathcal D$, implies that  $\hat{\D}=\{\hat{d}:d\in\D\}$ is an admissible multimetric on $X$. So, $\F$ is Krasnoselski\u\i\  contracting with respect to the admissible multimetric $\hat {\mathcal D}$ having cardinality $|\hat{\mathcal D}|\le|\mathcal D|\le \kappa$. If the multimetric $\mathcal D$ is sequentially complete, then so is the multimetric $\hat{\mathcal D}$ (as each Cauchy sequence in $(X,\hat{\mathcal D})$ remains Cauchy with respect to the multimetric $\mathcal D$).
\end{proof}

\subsection{Banach metrization of topologically contracting function systems}

Finally, we shall discuss the ``Banach'' version of Problem~\ref{pr6.1}. We refer the reader to the paper \cite{K} for a profound consideration of this problem restricted to spaces which are attractors of topologically contracting systems.

The following example constructed in \cite{BN} (see also \cite{NS}) indicates that this problem is not trivial even in the realm of compact metrizable spaces (cf. also examples from \cite{K}), and shows that Theorem~\ref{comp} cannot be strengthened by making $\F$  Banach contracting.

\begin{example}\label{ex11.1} There is a 1-dimensional Peano continuum $X$ (called ``shark teeth'') admitting a topologically contracting function system $\F$ which is Banach contracting for no admissible metric on $X$.
\end{example}

However we do not know if the function system $\F$ on the ``shark teeth'' from Example~\ref{ex11.1} is  Banach contracting for some admissible multimetric on $X$.

We will prove a result which states that the problem of a ``Banach'' remetrization of a function system $\F$ is equivalent to the problem of a ``Banach'' remetrization of some power $\F^m$, $m\in\N$, of $\F$. Note that our result is a particular version of \cite[Theorem 3]{J} (cf. also remetrization results from \cite{BV} and \cite{MM}), but obtained in a different way.

\begin{proposition} For a cardinal number $\kappa$ and a function system $\F$ on a Tychonoff space $X$ the following conditions are equivalent:
\begin{enumerate}
\item $\F$ is Banach contracting with respect to some admissible {(sequentially complete)} multimetric $\mathcal D$ on $X$ with  $|\mathcal D|\le\kappa$;
\item for some $m\in\IN$ the function system $\F^m$ is Banach contracting with respect to some admissible {(sequentially complete)} multimetric $\mathcal D$ on $X$ with $|\mathcal D|\le\kappa$.
\end{enumerate}
\end{proposition}

\begin{proof} The implication $(1)\Ra(2)$ is trivial. To prove that $(2)\Ra(1)$, assume that for some $m\in\IN$ the function system $\F^m$ is Banach contracting for some admissible multimetric $\mathcal D$ of cardinality $|\mathcal D|\le\kappa$. Then for every pseudometric $d\in\mathcal D$ there is a real number $\lambda<1$ such that $d\w_{\F^m}(t)\le\lambda t$ for all $t\in[0,\infty)$. Choose a real number $a>1$ such that $a^m\lambda<1$ and consider the pseudometric $\hat d$ on $X$ defined by
$$\hat d(x,y)=\sup_{n\in\w}\sup_{f\in\F^n}a^nd(f(x),f(y))\mbox{ for $x,y\in X$}.$$

To show that the pseudometric $\hat d$ is continuous, it suffices for every $x\in X$ and $\e>0$ to find a neighborhood $O_x\subset B_{\hat d}(x,2\e)$ of $x$. Consider the function system $\F^{<m}=\bigcup_{n<m}\F^n$. By the continuity of the maps $h\in\F^{<m}$, there is a neighborhood $O_x\subset X$ of $x$ such that $h(O_x)\subset B_d(h(x),\e/a^m)$ for every $h\in\F^{<m}$.

We claim that $O_x\subset B_{\hat d}(x,\e)$. Fix any $y\in O_x$, $n\in\w$ and $f\in\F^{n}$.
Write the number $n$ as $n=mq+r$ where $q\in\w$ and $0\le r<m$. Then $f=g\circ h$ for some functions $g\in(\F^m)^q$ and $h\in\F^{<m}$. Consequently,
$$
\begin{aligned}
a^n d(f(x),f(y))&=a^{mq+r}d(g\circ h(x),g\circ h(y))\le a^{mq+r}d\w_{\F^m}^q(d(h(x),h(y))\le\\
&\le a^{mq+r}\lambda^qd(h(x),h(y))<(a^m\lambda)^qa^r\e/a^m\le\e/a
\end{aligned}
$$and hence
$$\hat d(x,y)=\sup_{n\in\w}\sup_{f\in\F^n}a^n d(f(x),f(y))\le \frac{\e}a<\e,$$
so $O_x\subset B_{\hat d}(x,\e)$. Therefore the pseudometric $\hat d$ is continuous.

It follows from $\hat d\ge d$ for $d\in\mathcal D$ that $\hat{\mathcal D}=\{\hat d:d\in\mathcal D\}$ is an admissible multimetric on $X$, and if $\D$ is sequentially complete, then so is $\hat{\D}$.

Finally we show that each map $f\in\F$ is Banach contracting with respect to each pseudometric $\hat d\in\hat{\mathcal D}$. This follows from the upper bound
$$\hat d(f(x),f(y))=\sup_{n\in\w}\sup_{g\in\F^n}a^nd(g(f(x)),g(f(y))=\frac1a \sup_{n\in\w}\sup_{g\in\F^n}a^{n+1}d(g\circ f(x),g\circ f(y))\le \frac1a\,\hat d(x,y)$$
holding for any points $x,y\in X$.
\end{proof}

The following proposition yields a partial answer to the ``Banach'' version of Problem~\ref{pr6.1}.

\begin{proposition} Let $\F$ be a topologically contracting function system on a compact metrizable space $X$. If the attractor $A_\F$ of $\F$ is finite, then $\F$ is Banach  contracting with respect to some admissible metric $d$ on $X$.
\end{proposition}

\begin{proof} Choose a family $\{\bar O_x\}_{x\in A_\F}$ of pairwise disjoint closed neighborhoods of the points $x\in A_\F$ in the space $X$ and consider the quotient space $Y$ of the space $(X\times\{0\})\cup\bigcup_{x\in A_\F}\bar O_x\times[0,1]$ by the equivalence relation whose non-trivial equivalence classes are the sets $\bar O_x\times\{1\}$ with $x\in A_\F$. Therefore, $Y$ is the union of $X$ and the cones over the sets $\bar O_{x}$, $x\in A_\F$. Let $q:(X\times\{0\})\cup\bigcup_{x\in A_\F}{ \bar O_x}\times[0,1]\to Y$ be the quotient map.

Fix any admissible metric $\rho$ on the compact metrizable space $Y$. Choose any positive numbers $\alpha<\beta<1$ and for every $n\in\w$ find a real number $\hbar_n\in(0,1)$ such that $\diam_\rho q(\bar O_x\times[\hbar_n,1])<\alpha^{n}$ for all $n\in\w$. The topological contractivity of $\F$ guarantees that the sequence $(\F^n(X))_{n\in\w}$ converges to the attractor $A_\F$ in the Vietoris topology on the hyperspace $\K(X)$. Consequently, for some $n_0\in\w$ the set $\F^{n_0}(X)$ is contained in $\bigcup_{x\in A_\F}O_x$, where $O_x$ is the interior of $\bar O_x$ in $X$ for each $x\in A_\F$.

It is easy to construct a continuous function $h:X\to[0,1]$ such that $h^{-1}(1)=A_\F$, $h(X\setminus\bigcup_{x\in A_\F}O_x)\subset\{0\}$ and $h(\F^m(X)\setminus A_\F)\subset [\hbar_m,1)$ for all $m\ge n_0$. {Indeed, for fixed $x\in A_\F$ take a decreasing family of open neighborhoods $(U^x_n)$ such that $cl_X(U^x_1)\subset O_x$, $\cl_X(U^x_{n+1})\subset U^x_n$ and $\{x\}=\bigcap_{n\in\N}U^x_n$. Then let $(k_n)$ be an increasing sequence of positive integers such that $\{f(X):f\in\F^{k_n}\}\prec\{U^x_n:x\in A_\F\}$ and $k_1\geq n_0$ (we can take such sequence because we can assume that (by Theorem \ref{comp}) each $f\in\F$ is Matkowski contracting, and the diameter of $X$ is finite). Finally, for every $n\in\N$ let $g_n:X\to[0, \hbar_{k_{n+2}}-\hbar_{k_{n+1}}]$ be continuous and such that $g_n|(X\setminus (\bigcup_{x\in A_\F}U^x_{n}))=0$ and $g_n|(\bigcup_{x\in A_\F}cl_X(U^x_{n+1}))=\hbar_{k_{n+2}}-\hbar_{k_{n+1}}$. Also $g_0:X\to[0,\hbar_2]$ let be continuous and such that $g_0|(X\setminus (\bigcup_{x\in A_\F}O_x))=0$ and $g_0|\bigcup_{x\in A_\F}cl_X(U^x_1)=\hbar_2$. Then $h=\sum_{n\in\N}g_n$ satisfies our needs.}

The function $h$ induces the embedding $e:X\to Y$, $e:x\mapsto q(x,h(x))$ {($e$ is an embedding being a continuous injective map on a compact space)}, such that $\diam_\rho(e(O_x\cap\F^m(X))\le\diam_\rho q(\bar O_x\times[\hbar_m,1])\le \alpha^m$ for all $m\ge {n_0}$. Consider the metric $d$ on $X$ defined by $d(x,y)=\rho(e(x),e(y))$ and observe that $\diam_d(O_x\cap \F^m(X))\le\alpha^m$ for all $m\ge {n_0}$. Then the formula
$$\hat d(x,y)=\sup_{n\in\w}\sup_{f\in\F^n}\beta^{-{n}}d(f(x),f(y)),\;\;x,y\in X$$determines an admissible metric $\hat d\ge d$ on $X$ such that $\hat d(f(x),f(y))\le\beta \hat d(x,y)$ for any function $f\in\F$ and points $x,y\in X$. This means that the function system $\F$ is Banach contracting with respect to the metric $\hat d$.
\end{proof}

\begin{problem} Let $\F$ be a topologically contracting function system on a compact metrizable space $X$. Assume that $\F|A_\F$ is Banach  contracting with respect to some admissible metric on the attractor $A_\F$. Is there an admissible metric on $X$ making the function system $\F$ Banach  contracting?
\end{problem}

\end{document}